\documentclass[11pt,letterpaper]{article}%
\usepackage[utf8]{inputenc}
\usepackage[rflt]{floatflt}
\usepackage{fullpage}
\usepackage[T1]{fontenc}
\usepackage{graphicx}
\usepackage{epsfig}
\usepackage{color}
\usepackage{enumerate}
\usepackage{verbatim}
\usepackage{amsthm, amssymb}
\usepackage{amsmath}
\usepackage{thmtools,thm-restate}
\usepackage{wrapfig}
\usepackage{xspace}
\usepackage{amsfonts}
\usepackage{bbm}
\usepackage{algorithm, algpseudocode, caption, subcaption}

\usepackage[disable]{todonotes}

\definecolor{darkgreen}{rgb}{0,0.5,0}
\usepackage{hyperref}
\hypersetup{
    unicode=false,          %
    colorlinks=true,        %
    linkcolor=blue,          %
    citecolor=darkgreen,        %
    filecolor=magenta,      %
    urlcolor=cyan           %
}

\usepackage[disableredefinitions,roman]{complexity}
\usepackage[framemethod=tikz]{mdframed}
\global\mdfdefinestyle{myframe}{leftmargin=.75in,rightmargin=.75in,linecolor=black,linewidth=1.5pt,innertopmargin=10pt,innerbottommargin=10pt} 
\usepackage{mdframed}
\usepackage{framed}
\usepackage{nicefrac}

\date{}

\newtheorem{theorem}{Theorem}[section]
\newtheorem{lemma}[theorem]{Lemma}
\newtheorem{claim}[theorem]{Claim}
\newtheorem{infclaim}[theorem]{Informal Claim}

\newtheorem{definition}[theorem]{Definition}
\newtheorem{corollary}[theorem]{Corollary}

\newtheorem{condition}{Condition}

\usepackage[capitalize, nameinlink]{cleveref}
\usepackage{authblk}
\crefname{theorem}{Theorem}{Theorems}
\Crefname{lemma}{Lemma}{Lemmas}
\Crefname{alg}{Algorithm}{Algorithms}
\Crefname{claim}{Claim}{Claims}
\Crefname{infclaim}{Claim}{Claims}
\Crefname{observation}{Observation}{Observations}
\Crefname{invariant}{Invariant}{Invariants}
\Crefname{algorithm}{Algorithm}{Algorithms}

\newcommand{\cnt}{\rho}
\newcommand{\muh}{\hat{\mu}}

\newcommand{\Ex}{\mathbb{E}}

\renewcommand{\Pr}{\operatorname{{Pr}}}
\newcommand{\dfp}{\delta_{\textrm{false-pos}}}
\newcommand{\dfn}{\delta_{\textrm{false-neg}}}
\newcommand{\dpa}{\delta_{\textrm{param}}}
\newcommand{\dnfp}{\delta_{\textrm{no-false-pos}}}
\newcommand{\dnfn}{\delta_{\textrm{no-false-neg}}}
\newcommand{\spa}{S_{\textrm{params}}}
\newcommand{\conf}{\mathcal{C}_{\textrm{conf}}}
\newcommand{\Tls}{T_{\textrm{left-shift}}}
\newcommand{\Lls}{L_{\textrm{left-shift}}}
\newcommand{\Rls}{R_{\textrm{left-shift}}}
\newcommand{\other}{\operatorname{Other}}
\newcommand{\dfil}{d_{\textrm{filter}}}
\newcommand{\dest}{d_{\textrm{est}}}
\newcommand{\ssta}{\operatorname{Start}}
\newcommand{\send}{\operatorname{End}}
\newcommand{\rsmall}{R_{\textrm{small}}}
\newcommand{\rlarge}{R_{\textrm{large}}}
\DeclareMathOperator*{\argmax}{arg\,max}
\DeclareMathOperator*{\argmin}{arg\,min}

\allowdisplaybreaks[1]

\makeatletter
  \def\title@font{\Large\bfseries}
  \let\ltx@maketitle\@maketitle
  \def\@maketitle{\bgroup%
    \let\ltx@title\@title%
    \def\@title{\resizebox{\textwidth}{!}{%
      \mbox{\title@font\ltx@title}%
    }}%
    \ltx@maketitle%
  \egroup}
\makeatother
\title{Near-Optimal Mean Estimation with Unknown, Heteroskedastic Variances}
\author[]{Spencer Compton}
\author[]{Gregory Valiant}
\affil[]{Stanford University \authorcr
  \{\tt comptons, valiant\}@stanford.edu}
\begin{document}

\maketitle
\begin{abstract}
    Given data drawn from a collection of Gaussian variables with a common mean but different and unknown variances, what is the best algorithm for estimating their common mean? We present an intuitive and efficient algorithm for this task. As different closed-form guarantees can be hard to compare, the Subset-of-Signals model \cite{liang2020learning} serves as a benchmark for ``heteroskedastic'' mean estimation: given $n$ Gaussian variables with an unknown subset of $m$ variables having variance bounded by 1, what is the optimal estimation error as a function of $n$ and $m$? Our algorithm resolves this open question up to logarithmic factors, improving upon the previous best known estimation error by polynomial factors when $m = n^c$ for all $0<c<1$. 
    Of particular note, we obtain error $o(1)$ with $m = \tilde{O}(n^{1/4})$ variance-bounded samples, whereas previous work required $m = \tilde{\Omega}(n^{1/2})$. Finally, we show that in the multi-dimensional setting, even for $d=2$, our techniques enable rates comparable to knowing the variance of each sample.

\end{abstract}

\newpage

\section{Introduction}

Over the past decade, there has been a significant effort from the theoretical computer science and machine learning communities to reexamine fundamental learning and statistical estimation problems in non-i.i.d. settings.  Many of these efforts have focused on relaxing the independence assumption.  This includes the large body of work on \emph{robust statistics}, where a portion of the data are assumed to be drawn i.i.d. from a fixed distribution and no assumptions are made about the remainder of the data.  On the TCS side, work in robust statistics began by considering the problem of mean estimation in the Gaussian setting~\cite{diakonikolas2019robust,lai2016agnostic}, and then built up to considering more complex problems of learning or optimization (e.g~\cite{charikar2017learning,diakonikolas2019efficient}).   

Here, we instead consider the heterogeneous data setting, where samples are drawn independently, but from non-identical distributions.  Even for some of the most fundamental problems, such as the problem of mean estimation with Gaussian data that we consider, much is still unknown about both the information theoretic and computational landscapes in this heterogeneous but independent setting.  This is despite the practical importance of accurately extracting information from datasets whose contents have been gathered from heterogeneous sources (e.g. sourced from different workers, contributed by different hospitals or doctors, scraped from different websites, etc.).

Concretely, we consider the setting where we observe $n$ independent heteroskedastic (meaning having different variances) Gaussian random variables that have a common mean: $X_1 \sim N(\mu, \sigma_1^2),\ldots, X_n \sim N(\mu, \sigma_n^2)$, and our goal is to estimate their common mean, $\mu$.  Crucially, the variances $\sigma_i^2$ are \emph{unknown}.  This problem was explored in both the $d=1$ and higher dimensional settings in the work of Chierichetti, Dasgupta, Kumar, and Lattanzi \cite{chierichetti2014learning}.  In the case where the variances are known, the unbiased estimator that weights $X_i$ proportionally to $1/\sigma_i^2$ is easily shown to achieve optimal error $\Theta(1/\sqrt{\sum 1/\sigma_i^2})$ \cite{ibragimov1981statistical}\footnote{Theorem 3.1 of \cite{chierichetti2014learning} also contains a short proof of this.}.   When the variances are unknown, however, both the problem and the optimal rates seems to change fundamentally.   

In an effort to expose the core challenges of this problem, Liang and Yuan \cite{liang2020learning} introduced the Subset-of-Signals variant, parameterized by two numbers, $m,n$: as above, one observes $n$ independent Gaussian random variables with a common mean, $X_1,\ldots,X_n$, with the assumption that $m$ have variance at most 1, and one makes no assumptions about the variances of the remaining $n-m.$  Our results address the more general formulation, though are easier to interpret in this Subset-of-Signals setting, for which our approach achieves the known lower bounds, up to logarithmic factors.

\subsection{Related Work}

As mentioned above, this problem of heteroskedastic mean estimation was 
considered by Chierichetti, Dasgupta, Kumar, and Lattanzi in the $d=1$ dimensional and (isotropic) high dimensional setting where $X_i \sim N(\mu,\sigma_i^2I)$~\cite{chierichetti2014learning}.  Note that in this formulation, mean estimation becomes easier as $d$ becomes larger, as there is more information with which to infer the values of $\sigma_i$. Thus, while independent, the observations in different dimensions are often called ``entangled.'' When $d = \Omega(\log(n))$, \cite{chierichetti2014learning} attain estimation error of $\Theta\left(\sqrt{\frac{1}{\sum_{i=2}^n \frac{1}{\sigma_i^2}}}\right)$ for each dimension with high probability. Note that this is nearly identical to the classical known-variance rate, other than missing the dependence on $\sigma_1$. These results prompted subsequent work to focus on the more challenging small dimensional or one-dimensional settings for which it is more difficult or impossible to accurately recover the $\sigma_i$'s.

In the one-dimensional setting, \cite{chierichetti2014learning} attains a guarantee with respect to the $O(\log(n))$ smallest $\sigma_i$, giving an algorithm with expected error $\Ex[|\mu - \muh|] = \min\limits_{2 \le k \le \log(n)} \tilde{O}(n^{1/2 (1 + 1/(k-1))} \sigma_k)$. Moreover, they showed lower bounds that demonstrated how the known-variance rates can be polynomially better than an optimal estimator that does not know the variances. 

Subsequent works, ~\cite{pensia2019estimating,pensia2022estimating,devroye2020mean,devroye2023mean,liang2020learning,yuan2020learning}, which we discuss below, improve upon this in various regimes: their upper and lower bounds in the case of the Subset-of-Signals setting, together with our results, are depicted in Figure~\ref{fig:subset}.

The work of Pensia, Jog, and Loh (preliminarily \cite{pensia2019estimating} and later \cite{pensia2022estimating}) develops machinery for analyzing the performance of classic estimators in this setting: the modal estimator, $k$-closest estimator, and the median. 
Using this, they show guarantees for a hybrid estimator and give complementary lower bounds that illustrate how under some conditions on $\sigma_1,\dots,\sigma_n$ their estimator is near-optimal.\footnote{We later observe in \cref{fig:subset} that its guarantees can be polynomially suboptimal in a natural setting.} They also investigate the setting of heteroskedastic linear regression, as well as showing guarantees for their algorithm in $d>1$ dimensions. Moreover, their results generalize from Gaussian distributions to radially symmetric and unimodal distributions.

The work of Devroye, Lattanzi, Lugosi, and Zhivotovskiy (preliminarily \cite{devroye2020mean} and later \cite{devroye2023mean}) also develops tools for sharp analysis of the sample median and modal estimator. In order to provide an adaptive algorithm requiring no parameter tuning, they employ subroutines that yield confidence intervals which they eventually intersect. Our algorithm will utilize a similar paradigm of intersecting confidence intervals obtained by (different) subroutines.

The works of Liang and Yuan \cite{liang2020learning,yuan2020learning} provide estimation guarantees for the iterative truncation algorithm (a widely used heuristic). Importantly, they also introduce the Subset-of-Signals model, where $m$ samples have variance bounded by $1$, and it is desired to know the optimal estimation guarantee as a function of $n$ and $m$. This framing is particularly helpful because the closed-form guarantees of various related work can otherwise be difficult to directly compare. In \cref{fig:subset}, we show the guarantees of related work in terms of the Subset-of-Signals model. Finally, Liang and Yuan show lower bounds for the optimal estimation error in this model.

\begin{figure}[ht]
    \centering
    \includegraphics[scale=0.2]{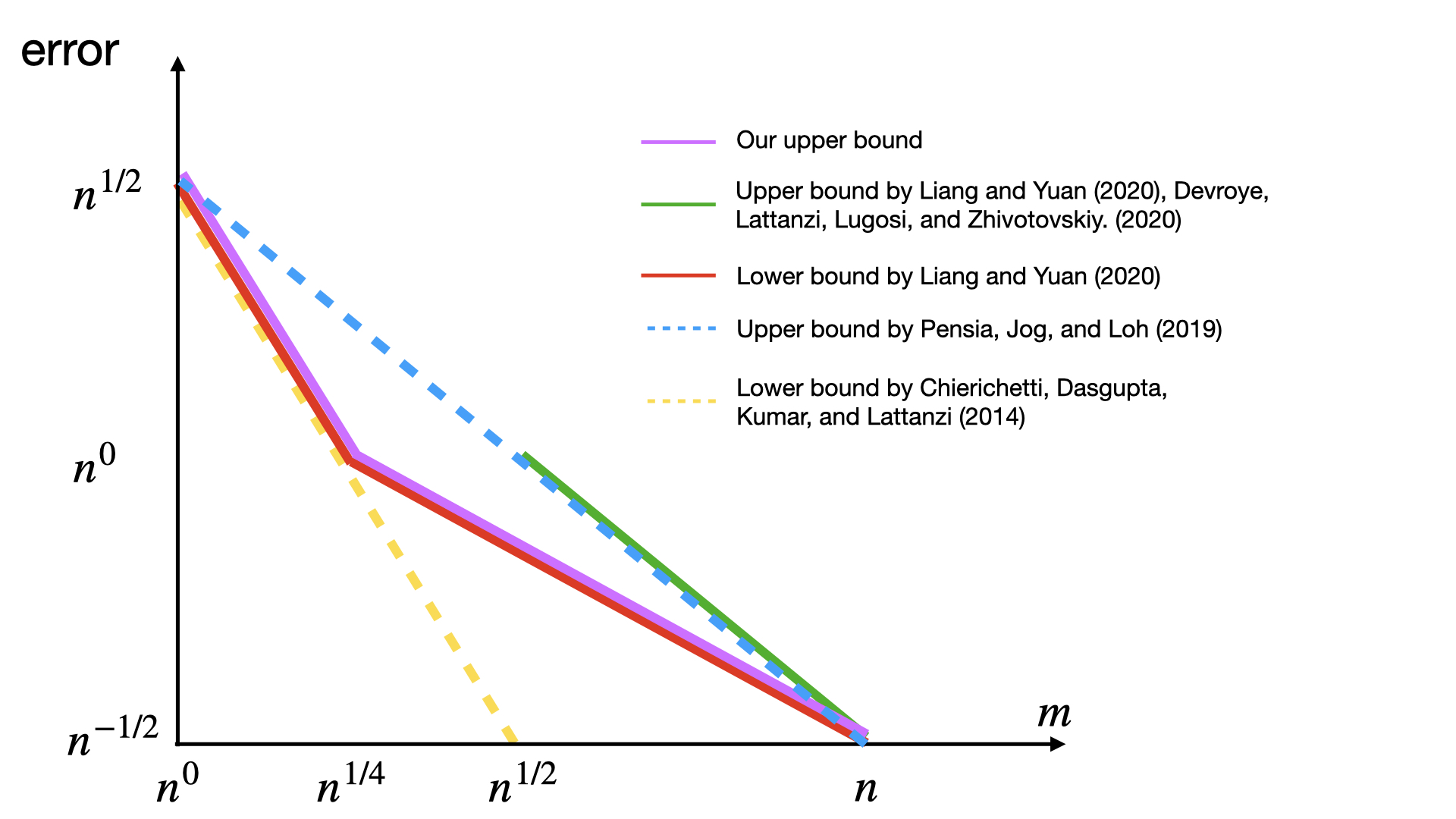}
    \caption{Guarantees of our upper bound and those of prior work for mean estimation in the Subset-of-Signals model, where one observes $n$ independent Gaussian random variables with a common mean, and an unknown subset of $m \le n$ samples have variance at most $1$, with no assumptions on the variance of the remaining $n-m$.  The $x$-axis denotes $m$, the number of samples with variance bounded by $1$, and the $y$-axis denotes the estimation error. Our upper bound matches the known lower bound up to logarithmic factors, and improves the estimation error by polynomial factors when $m=n^c$ for $0<c<1$.  (Figure based on plot from  \cite{liang2020learning}).\label{fig:subset}}
    
\end{figure}

\paragraph{Related Work Beyond Heteroskedastic Mean Estimation:}  There have been several lines of work exploring property testing, estimation, and learning in settings with independent, but non-identical samples.  These models span a large spectrum in terms of how much heterogeneity is present, relative to the sample size.  On one extreme, there is a large volume of work on learning \emph{mixture models} (of Gaussians, linear regressions, etc., see e.g. \cite{belkin2010polynomial,moitra2010settling,diakonikolas2023sq,hopkins2018mixture,pmlr-v119-kong20a}).  Typically, in these settings there are a small number (often just a constant number) of distributions, and each datapoint is drawn i.i.d. from one of these.   Comparatively fewer works explore the other extreme, where a single sample (or small batch of samples) is drawn from each distribution---typically too little to learn the distribution---and the goal is to estimate some property of the set of distributions.  This includes the property testing work of Levi et al.~\cite{levi2013testing}, and work on estimating properties of populations of parameters, such as estimating the multiset of coin biases given a small number of tosses of each coin (e.g.~\cite{tian2017learning,vinayak2019maximum}).

\newpage
\subsection{Our Contributions}

In our work, we design new algorithms for heteroskedastic mean estimation with polynomially-better error guarantees than prior work, explicitly answering the open problem of \cite{liang2020learning} (see \cref{fig:subset}):

\begin{center}
    \emph{Given samples of $n$ independent Gaussians with a common mean, and with an unknown subset of $m$ samples having variance bounded by 1, what is the best possible estimation error?}
\end{center}

\begin{restatable}[Optimal Subset-of-Signals]{theorem}{theoremsubset}\label{theorem:sub-opt}
     Consider observing $n$ Gaussian samples with a common mean $X_i \sim N(\mu, \sigma_i^2)$, where $\sigma_1 \le \dots \le \sigma_m \le 1$, the variances are unknown to the algorithm, and samples are presented in an arbitrary order. For any constant $\delta$, there exists a constant $C$ such that with probability at least $1 - \frac{1}{n^{\delta}}$, \cref{algo:estimate} attains:

     \begin{itemize}
         \item $\tilde{O}\left( \frac{n}{m^4} \right)^{1/2}$ error if $ C \log(n) \le m \le n^{1/4}$
         \item $\tilde{O}\left( \frac{n}{m^4} \right)^{1/6}$ error if $n^{1/4} \le m \le n$
     \end{itemize}
\end{restatable}

As our algorithm is scale-invariant and translation-invariant, this also enables the closed-form:

\begin{corollary}
Consider observing $n$ Gaussian samples with a common mean $X_i \sim N(\mu, \sigma_i^2)$, where $\sigma_1 \le \dots \le \sigma_n$, the variances are unknown to the algorithm, and samples are presented in an arbitrary order. For any constant $\delta$, there exists a constant $C$ such that with probability at least $1 - \frac{1}{n^{\delta}}$, \cref{algo:estimate} attains error

    \begin{center}
        $\tilde{O} \left( \min \left(\min\limits_{C \log(n) \le i \le n^{1/4}} \sigma_i \cdot \left( \frac{n}{i^4} \right)^{1/2},  \min\limits_{n^{1/4} \le i \le n} \sigma_i \cdot \left( \frac{n}{i^4} \right)^{1/6} \right) \right)$
    \end{center}
\end{corollary}

Our techniques also naturally extend to the $d >1$ dimensional setting,  resolving the implicit open problem of \cite{chierichetti2014learning}:   \emph{How large does the dimension $d$ need to be to nearly attain the error rate that would be achievable if the variances were known?}  We show that even when $d=2$, this known-variance rate can nearly be attained, improving upon the prior guarantee of \cite{chierichetti2014learning} that required $d = \Omega(\log(n))$:

\begin{restatable}{theorem}{theoremestmultiboth}\label{theorem:est-multi-both}

Consider observing $n$ 2-dimensional Gaussian samples with a common mean $X_i \sim N(\mu, \sigma_i^2 I)$, where $\sigma_1 \le \dots \le \sigma_n$, the variances are unknown to the algorithm, and samples are presented in an arbitrary order. With probability $1-o(1)$, \cref{algo:adjust-multi} attains error $\tilde{O}\left(\sqrt{\frac{1}{\sum_{i=2}^n \frac{1}{\sigma_i^2}}}\right)$.
\end{restatable}

\subsection{Preliminaries}

Let $\cnt(l,r)$ denote the random variable corresponding to the number of samples with value $ \in [l,r]$. $f_\mathcal{D}(\cdot)$ is the density function of distribution $\mathcal{D}$. For $d=1$, in instances where we must refer to the samples in order of realized value, we refer to them by $Y_1 \le \dots \le Y_n$. Meaning, we realize $X_1, \dots, X_n$ with $X_i \sim N(\mu, \sigma_i^2),$ and observe $Y_1 \le \dots \le Y_n$ where the $Y_i$'s are the $X_i$'s sorted in non-decreasing order.

\section{Overview of Our Techniques}

In this section, we provide the high-level intuition for our approach and results, and describe the key lemmas that facilitate our analysis. Finally, we discuss how our approach and analysis can be furthered to attain results for multi-dimensional heteroskedastic mean estimation.

\subsection{Intuition and Existing Estimators}

As discussed earlier, mean estimation and even heteroskedastic mean estimation has been studied by a variety of prior works that leverage different algorithmic ideas. 
Here, we provide a brief overview to give  intuition into the challenges of the problem, and motivate our main algorithmic ideas. %

The two most basic estimators are the \emph{empirical mean} and the \emph{empirical median}.  Neither of these, however, adequately leverage the heterogeneity in the quality of samples in settings where some variances are significantly larger than others.  In the case of returning the empirical mean, $\frac{X_1 + \dots X_n}{n},$ even if all but one sample has variance $1$ and a single sample has arbitrarily large variance, the empirical mean also will have large variance.  While the \emph{median} of the $X_i$'s has some robustness to such settings, it also fails to leverage heterogeneity---this is especially easy to see in the fact that the median is blind to settings where $\ll \sqrt{n}$ samples have significantly smaller variance than the rest.  For example, suppose $X_1, \dots, X_{n^{1/2-\varepsilon}} \sim N(\mu,1)$ and $X_{n^{1/2-\varepsilon}}, \dots, X_n \sim N(\mu, \infty)$.  The median will incur unbounded expected error, while alternative algorithms, such as one that looks for the tightest cluster of $n^{1/2-\varepsilon}$ points and then takes the average of the cluster, would incur expected error of $\Theta(\frac{1}{\sqrt{n^{1/2 - \varepsilon}}})$.

These settings where there are a small number of very good samples motivate creating estimators that search for tightly-clustered sets of samples, and return a statistic of the samples in the cluster. This intuitively reflects that if there are few low-variance samples, we would prefer our estimate to rely almost purely on those good samples \emph{if we could identify them}. The \textbf{$k$-closest estimator}, and the \textbf{``modal''} estimator are two estimators that leverage this intuition.  The $k$-closest estimator looks at the $k$-closest points and returns their midpoint.  The ``modal'' estimator returns the value $\muh$ containing the most samples within $[\muh-w,\muh+w]$.  The parameters $k$ and $w$ are chosen so as to isolate an appropriate scale that focuses on the high-quality samples. As one might expect, these estimators are quite similar, and there is nearly a bijection between the $k$-closest estimator and the modal estimator with parameter $w = \argmin_{w} (\max_{\muh} \cnt(\muh-w,\muh+w) \ge k)$. These estimators have 
been at the core of the previously-best guarantees for heteroskedastic mean estimation. Despite this, their shortcomings are illustrated even in the homoskedastic case where all samples have equal variance: when all samples $X_1, \dots, X_n \sim N(\mu,1)$ there is no choice of $k$ or $w$ for which the modal or $k$-closest estimators yield expected error better than $\Theta(n^{-1/3})$ \cite{chernoff1964estimation,kim1990cube}, despite expected error $O(1/\sqrt{n})$ being achievable by the mean or median.\footnote{For variants of the $k$-closest estimator that return the mean or median of the $k$-closest points, rather than their midpoint, this can behave similarly to the mean or median for sufficiently large $k$, although they are still suboptimal in the heteroskedastic case.}

\subsection{A ``Balanced'' Modal Estimator}
At its core, our estimator behaves similarly to a modal estimator, that returns the estimate $\muh$ which maximizes the number of samples in the range $[\muh-w,\muh+w]$, with the additional condition that this range be ``balanced'' in the sense that the number of samples in the interval $[\muh-w, \muh]$ is approximately the same as the number of samples in the interval $[\muh, \muh+w].$  

Before discussing how $w$ is chosen, we describe the intuition for this balanced condition.  Returning to the homoskedastic case where all variances are 1, suppose we are trying to decide whether to return the true mean, $\mu$, versus a slightly offset version of it, $\mu + \Delta.$  The standard modal estimator with parameter $w=1$ is  trying to decide whether there is more probability mass in the interval $[\mu -1,\mu+1]$ versus the interval $[\mu -1 + \Delta,\mu+1 + \Delta]$. This depends on the difference between the mass in the intervals $[\mu - 1, \mu - 1 + \Delta]$ and $[\mu + 1, \mu + 1 + \Delta]$. The difference in expectation is roughly the derivative of the probability density function of the standard Gaussian, evaluated at $1$ times the \emph{square} of $\Delta$, namely  $O(\Delta^2 n)$, while the standard deviation of the difference is roughly $O(\sqrt{\Delta n})$. The signal of the true mean overpowers the variance when $\Delta \gg n^{-1/3}$, matching classical guarantees for the modal estimator.
 In contrast, when evaluating the balance condition at $\mu+\Delta,$ the relevant quantity is the difference between the densities in the intervals $[\mu + \Delta -1, \mu+\Delta]$ and $[\mu + \Delta, \mu+\Delta+1]$. The difference in expectation is roughly the difference in the standard Gaussian density in the interval $[\mu,\mu+\Delta]$ and the interval $[\mu + 1, \mu+1+\Delta]$. In particular, this quantity is \emph{linear} in the offset $\Delta$, as opposed to quadratic. We obtain a difference in expectation that is $O(\Delta n)$, while the standard deviation is $O(\sqrt{n})$. Hence, we can detect imbalance when $\Delta \gg n^{-1/2}$, yielding more accurate estimates that match the best guarantees for homoskedastic estimation.

This ``balanced'' modal estimator attains nearly-optimal error for homoskedastic mean estimation in a way that seems amenable to zooming into scales that would leverage heteroskedasticity, unlike the median or mean. We will see that  (perhaps surprisingly), this balanced modal estimator can also provide a near-optimal estimator from \emph{heteroskedastic} observations if the perfect width $w$ to use was known. To address this caveat that we do not know which width, $w$, to use, we propose a similarly-intuited approach we call the \emph{balance-finding algorithm}. Oversimplifying, this algorithm will enable us to accomplish something similar to looking for the information of the balanced modal estimator at multiple scales of $w$ simultaneously.

\textbf{Balance finding.} Our primary algorithmic technique is to search for the phenomenon of a particular kind of \emph{balance} that implies a high-probability confidence interval for the mean. We will look for such balance at many scales (similar to trying many values of $w$) and intersect our obtained confidence intervals to determine our final estimate. To illustrate this phenomenon, consider counting the number of samples that are slightly less than $\mu$, and the number of samples slightly larger than $\mu$. If we use ``slightly'' to mean within an interval of size $w$, we are considering $\cnt(\mu-w,\mu)$ and $\cnt(\mu,\mu+w)$ respectively (recall that $\cnt(l,r)$ denotes the number of samples within $[l,r]$). Naturally, as our density is symmetric, we expect $\cnt(\mu-w,\mu) \approx \cnt(\mu,\mu+w)$, meaning these terms are $\tilde{\Theta}(\sqrt{\cnt(\mu-w,\mu+w)})$ apart. For appropriately chosen $w$ and any estimate $\muh$, an observation that $\cnt(\muh-w,\muh)\approx\cnt(\muh,\muh+w)$ can be roughly interpreted as evidence that either $|\mu - \muh|$ is small, or that $[\muh-w,\muh+w]$ corresponds to a relatively flat region of the density curve.

This illuminates the desire to distinguish between estimates near $\mu$ and estimates far from $\mu$ but in flat regions of the density curve. Intuitively, in the case that our estimate is merely in a flat region, we expect to still see this balance if we perturb our estimate. More concretely, suppose we perturb our flat-region $\muh$ by a term $\Delta$, we still expect to see $\cnt((\muh+\Delta)-w,\muh+\Delta) \approx \cnt(\muh+\Delta,(\muh+\Delta)+w)$. On the other hand, we do not expect to see this balance when $\muh$ is near $\mu$. If we move our estimate $\Delta$ to the left then we expect to see many more samples to its right, or $\cnt((\muh-\Delta)-w,\muh-\Delta) \ll \cnt(\muh-\Delta,(\muh-\Delta)+w)$. Similarly, if we move the estimate $\Delta$ to the right we expect $\cnt((\muh+\Delta)-w,\muh+\Delta) \gg \cnt(\muh+\Delta,(\muh+\Delta)+w)$. This motivates searching for a meaningful type of balance, where the balance is not observed for the perturbed estimates, and thus resembling the case where $|\mu - \muh|$ is small. 

Finding balance can be defined with respect to the estimate $\muh$, the perturbation $\Delta$, the width $w$, and a confidence parameter that determines thresholds for $\approx,\ll,\gg$ as used above. In this section, assume the confidence parameter is defined such that the probability of ever finding a false-positive meaningful balance is inverse-polynomially small. We will then more precisely describe a balance as a $(w,\Delta,\muh)$-balance. We claim that, with high probability, there will be no $(\cdot,\Delta,\muh)$-balance where $|\mu - \muh| > \Delta$: yielding a confidence interval of $\mu \in [\muh - \Delta, \muh + \Delta]$. Accordingly, our strategy is to test many carefully-chosen tuples of $(w,\Delta,\muh)$-balance and intersect the confidence intervals we obtain. In \cref{algo:test}, we outline our subroutine for testing a $(w,\Delta,\muh)$-balance.

\begin{algorithm}[t]
    \caption{Testing $(w,\Delta,\muh)$-balance} \label{algo:test}
    \hspace*{\algorithmicindent} 
    \begin{flushleft}
      {\bf Input:} width $w$, shift $\Delta$, and potential mean $\muh$  \\
      {\bf Output:} PASS (it likely holds that the true mean $\mu \in [\muh-\Delta,\muh+\Delta]$), or FAIL (insufficient evidence or evidence against  $\mu \in [\muh-\Delta,\muh+\Delta]$)\\
      {\bf Description:} This test will PASS if the number of samples in the intervals $[\muh-w,\muh]$ and $[\muh,\muh+w]$ are approximately equal, yet after shifting these intervals by $\pm\Delta$ the halves become significantly unbalanced (evidencing a higher density of samples near $\muh$ versus $\muh \pm w$). 
    \end{flushleft}
    \begin{algorithmic}[1]
    
    \Procedure{Test}{$w,\Delta,\muh$}:
    \State $L_{\textrm{shift-right}} \gets \cnt(\muh + \Delta - w,\muh + \Delta)$ \Comment{Count samples within $[\muh + \Delta - w, \muh + \Delta]$.}
    \State $R_{\textrm{shift-right}} \gets \cnt(\muh+\Delta,\muh + \Delta + w)$  \Comment{Count samples within $[\muh+\Delta,\muh + \Delta + w]$.}
    \State $T_{\textrm{shift-right}} \gets \cnt(\muh+\Delta - w,\muh + \Delta + w)$ \Comment{Count samples within $[\muh+\Delta - w,\muh + \Delta + w]$.}
    
    \If{$L_{\textrm{shift-right}} - R_{\textrm{shift-right}} \le \sqrt{C_{\dfp} \log(n) T_{\textrm{shift-right}}}$ \textbf{or} $T_{\textrm{shift-right}} < C_{\dfp} \log(n)$} 
    
            \Return FAIL 
    \EndIf

    \State $L_{\textrm{shift-left}} \gets \cnt(\muh-\Delta - w,\muh - \Delta)$ \Comment{Count samples within $[\muh-\Delta - w,\muh - \Delta]$.}
    \State $R_{\textrm{shift-left}} \gets \cnt(\muh - \Delta,\muh - \Delta + w) $ \Comment{Count samples within $[\muh - \Delta,\muh - \Delta + w]$.}
    \State $T_{\textrm{shift-left}} \gets \cnt(\muh - \Delta - w,\muh - \Delta + w) $ \Comment{Count samples within $[\muh - \Delta - w,\muh - \Delta + w]$.}
    
    \If{$R_{\textrm{shift-left}} - L_{\textrm{shift-left}} \le \sqrt{C_{\dfp} \log(n) T_{\textrm{shift-left}}}$  \textbf{or} $T_{\textrm{shift-left}} < C_{\dfp} \log(n)$} 
    
            \Return FAIL \label{line:shift-left-fail} 
    \EndIf

    \Return PASS 

    \EndProcedure
    \end{algorithmic}
\end{algorithm}

What remains is to design an algorithm that tests the correct balances that yield sufficiently small and correct confidence intervals. Algorithmically, we remark that for a given $w$ and $\Delta$, we can use a sweep-line method to find all ranges of $\muh$ where there exists $(w,\Delta,\muh)$-balance in $\tilde{O}(n)$ time. Thus, we may obtain an $\tilde{O}(n)$ time algorithm if we can select $\tilde{O}(1)$ pairs of $(w,\Delta)$ to consider, and can show that testing just balances with these parameters will obtain our desired estimation error. While we do not fully motivate it until later, we provide our approach in \cref{algo:estimate}.

\begin{algorithm}[t]
    \caption{Estimation-Algorithm} \label{algo:estimate}
    \hspace*{\algorithmicindent} 
    \begin{flushleft}
      {\bf Input:} $Y_1 \le  \dots \le Y_n$ \\
      {\bf Output:} Range $\conf$ (can choose any arbitrary value in this range as the estimate $\muh$) \\
    \end{flushleft}
    \begin{algorithmic}[1]
    \Procedure{Sweep-Test}{$w,\Delta$}: \label{line:sweep}

    \Return $S_{w,\Delta}$
    \Comment{Returns set $S_{w,\Delta}$ of $O(n)$ intervals of $\muh$ that PASS $\operatorname{Test}(w,\Delta,\muh)$}
    \EndProcedure

    \Procedure{Generate-Tests}{$Y_1 \le \dots \le Y_n$}: \label{line:params}
    \State $\spa \gets \{ \infty \}$
    \For{$i \in [\lfloor \log(n) \rfloor]$}
        \State{$r_{2^i} \gets \min_{j} Y_{j + 2^i} - Y_j$}
        \Comment{$r_{2^i}$ is the gap between the closest $2^i$ samples}
        \For{$j \in \{ - \lceil C_{\dpa} \log(n) \rceil , \dots, \lceil C_{\dpa} \log(n) \rceil$}
            \State $\spa \gets \spa \cup r_{2^i} \cdot 2^j$
            \Comment{Approximating $\sigma_{2^i}$ by powers of 2 near $r_{2^i}$.}
        \EndFor
    \EndFor
    
    \Return $\spa$
    \Comment{Returns $\spa$, including $\infty$ and approximations of $\sigma_{2^i}$}
    \EndProcedure
    
    \Procedure{Estimation-Algorithm}{$Y_1 \le \dots \le Y_n$}:
    \State $\conf \gets [-\infty, \infty]$
    \Comment{Interval we are confident $\mu$ is within}
    \State $\spa \gets \operatorname{Generate-Tests(Y)}$
    \Comment{Determine values of $w,\Delta$}
    \For{$w,\Delta \in \spa$}
        \State $S_{w,\Delta} \gets \operatorname{Sweep-Test(w,\Delta})$ 
        \Comment {Values of $\muh$ that Pass $\operatorname{Test}(w,\Delta,\muh)$.}
        \If{$S_{w,\Delta} \ne \emptyset$}
            \State $\conf \gets \conf \cap \min_{\muh \in S_{w,\Delta}} [\muh - \Delta, \muh + \Delta]$
            \State $\conf \gets \conf \cap \max_{\muh \in S_{w,\Delta}} [\muh - \Delta, \muh + \Delta]$ \Comment{Intersect confidence intervals.}
        \EndIf
    \EndFor
    
    \Return $\conf$ \Comment{Can estimate $\muh$ as any arbitrary value in $\conf$.}
    \EndProcedure
    \end{algorithmic}
\end{algorithm}

\subsection{Analyzing Estimation Error}\label{sub:analyze-est}

\textbf{Near-optimal guarantees for simplified Subset-of-Signals.} We will now informally show that finding balance is sufficient for obtaining near-optimal guarantees in a simplified version of the Subset-of-Signals model where at least $m$ samples have $\sigma_i \le 1$, \emph{and the remaining samples all have the same value of $\sigma_i = \sigma^*$ (this additional assumption is only to permit a cleaner explanation here)}. More sophisticated techniques will later enable us to show the same guarantees for (unsimplified) Subset-of-Signals, and results for more general settings.

The existence of $(w,\Delta,\mu)$-balance will typically imply that our algorithm obtains $O(\Delta)$ error with high probability. This will follow from showing that: (i) with high probability there is no $(\cdot,\Delta',\muh)$-balance where $|\mu - \muh| > \Delta'$, and (ii) our algorithm will test sufficiently similar tuples that find a $(\cdot, \Delta', \cdot)$-balance with $\Delta' = O(\Delta)$. Accordingly, if there exists a $(w,\Delta,\mu)$-balance, then we expect our algorithm to find a balance yielding a correct confidence interval of width $O(\Delta)$ containing $\mu$. This motivates our focus on studying the conditions under which $(w,\Delta,\mu)$-balance exists:

\begin{infclaim}\label{claim:findit}
    $(w,\Delta,\mu)$-balance will exist with high probability if $\mathbb{E}[\cnt(\mu,\mu+\Delta) - \cnt(\mu+w,\mu+w+\Delta)]^2 \ge C_1 \log(n) \cdot \mathbb{E}[\cnt(\mu,\mu+w )]$.
\end{infclaim}

This follows from how the imbalance after shifting will be much larger than the standard deviation of the difference between correctly-balanced halves centered at $\mu$. We will use the simple condition of \cref{claim:findit} to obtain desired estimation error. As seen in \cref{fig:subset}, the optimal rate for Subset-of-Signals undergoes a phase transition at $m=n^{1/4}$. We obtain this rate up to logarithmic factors:

\begin{lemma}\label{lemma:toy-large-n}
    When $m \in [n^{1/4},n]$, with high probability there exists a $(w,\Delta,\mu)$-balance with $\Delta=\tilde{O}\left( \frac{n}{m^4} \right)^{1/6}$.
\end{lemma}
\begin{proof}
    We will consider evaluating two types of balance, and conclude that at least one of these balances must exist with the desired $\Delta$. 
    
    By \Cref{claim:findit}, we can find $(1,\Delta,\mu)$-balance if $(m \cdot \Delta)^2 \ge O(1) \cdot C_1 \log(n) \cdot \Ex[\cnt(\mu-1,\mu+1)]$. Meaning, if we do not find such balance, $\Delta \le O(1) \cdot \sqrt{\frac{C_1 \log(n) \Ex[\cnt(\mu - 1, \mu + 1)]}{m^2}}$.

    Intuitively, if this is an undesirable bound on $\Delta$, then $\Ex[\cnt(\mu - 1, \mu + 1)]$ must be large, meaning many of the $n-m$ samples of standard deviation $\sigma^*$ must be realized in $[-1,+1]$, and thus $\sigma^*$ must not be too large. In other words, either we are able to find balance from our $m$ ``good'' points, or our remaining $n-m$ ``bad'' points must not actually be too bad. For our other type of balance, we will notice how $(\infty,\Delta,\mu)$-balance behaves similarly to classical high-probability guarantees for the median. We will find such a balance if $\Ex[\cnt(\mu-\Delta,\mu+\Delta)] \ge O(1) \cdot \sqrt{C_1 \log(n) n}$.

    Combining both restrictions, if we cannot find either balance then $\Delta \le O(1) \cdot \sqrt{\frac{C_1 \log(n) \cnt(\mu - 1, \mu + 1)}{m^2}} \le O(1) \cdot \sqrt{\frac{C_1 \log(n) \Ex[\cnt(\mu - \Delta, \mu + \Delta)]}{\Delta m^2}} \le O(1) \cdot \sqrt{\frac{C_1 \log(n) \sqrt{C_1 \log(n) n}}{\Delta m^2}}$. This implies $\Delta \le O(1) \cdot (C_1^{3/2} \log^{3/2}(n))^{1/3} \cdot (\frac{n}{m^4})^{1/6} = O(\sqrt{\log(n)} \cdot (\frac{n}{m^4})^{1/6}) = \tilde{O}((\frac{n}{m^4})^{1/6}))$.
    
\end{proof}

\begin{lemma}\label{lemma:toy-small-n}
    When $m \in [C' \log(n),n^{1/4}]$, with high probability there exists a $(w,\Delta,\mu)$-balance with $\Delta=\tilde{O}\left( \frac{n}{m^4} \right)^{1/2}$.
\end{lemma}
\begin{proof}
    We will again consider evaluating a pair of balances, and conclude that at least one of these values must exhibit balance with the desired $\Delta$. 
    
    By \Cref{claim:findit}, we can find $(1,\frac{1}{2},\mu)$-balance if $m^2 \ge O(1) \cdot  C_1 \log(n) \cdot \Ex[\cnt(\mu-1,\mu+1)]$. Since our guarantees for $\Delta$ in this lemma are super-constant, finding this balance would be sufficient. If we do not find such balance, then $m^2 \le O(1) \cdot C_1 \log(n) \cdot \Ex[\cnt(\mu-1,\mu+1)] \le O(1) \cdot  C \log(n) \frac{n}{\sigma^*} \implies \sigma^* \le O(1) \cdot \frac{C n \log(n)}{m^2}$.

    Similar to \cref{lemma:toy-large-n}, our inability to find balance from the $m$ samples implies $\sigma^*$ cannot be too large. We will then find the median-like balance of $(\infty,\Delta,\mu)$-balance if $\Ex[\cnt(\mu-\Delta,\mu+\Delta)] \ge \sqrt{C_1 \log(n) n}$. Finally, this implies we find $(\infty,\Delta,\mu)$-balance for a $\Delta \le O(1) \cdot \frac{\sqrt{C_1 \log(n) n}}{\Ex[\cnt(\mu -1, \mu+1)]} \le O(1) \cdot \frac{\sigma^* \sqrt{\log(n)}}{\sqrt{n}} \le O(1) \cdot \frac{\sqrt{n} \log^{1.5}(n)}{m^2} = \tilde{O}((\frac{n}{m^4})^{1/2})$.

\end{proof}

Accordingly, one may obtain desired rates for simplified Subset-of-Signals by just testing the collection of tuples we discussed in the proofs of \cref{lemma:toy-large-n,lemma:toy-small-n}.

\textbf{Additional considerations.} We will need additional non-trivial considerations for proving our unsimplified results. Some include:

\emph{(Unsimplified) Subset-of-Signals. } If the $n-m$ remaining samples are allowed to have any value of $\sigma_i$, then checking just the tuples of balances in \cref{lemma:toy-large-n,lemma:toy-small-n} will not be sufficient to find the desired balance. This is roughly because there may be groups of $\sigma_i$ that interfere with balance at the scale of $1$, while still not helping produce a good median. With some nuance, we later show (i) there still must exist some scale at which to find desired balance, and (ii) we can choose a set of $\tilde{O}(1)$ tuples which will test something sufficiently close to discover said desired balance.

\emph{Choosing testing tuples. } The previous point touches on how we require some way of testing the correct collection of balances. Moreover, it would be desirable if our estimator was scale-invariant so that if $m$ samples have $\sigma_i \le \nu$, then we could attain the analogous Subset-of-Signals guarantee scaled by $\nu$. One may expect that if we are looking for balance driven by $k$ good samples, the correct $\Delta$ and $w$ to test may be within a polynomial factor of the distance between the $k$-closest points ($r_k$). Later, we will show it is sufficient to consider pairs of $w$ and $\Delta$ that are powers of $2$ and polynomially-close to a $r_{2^i}$ for $i \in [1,\log(n)]$, giving $\tilde{O}(1)$ tuples to test in a scale-invariant manner.

\subsection{Multi-Dimensional Estimation}\label{sec:over-high}
In this section, we focus on estimation with $d$-dimensional observations. Each $X_i \sim N(\mu,\Sigma_i)$, where $\mu$ is a $d$-dimensional vector and $\Sigma_i$ is a $d \times d$ covariance matrix. If each $\Sigma_i$ can be an arbitrary diagonal covariance matrix, then observations in different dimensions are unrelated and thus there is nothing possible beyond considering $d$ independent instances of 1-dimensional estimation. However, if $\Sigma_i = \sigma_i^2 I$, then each sample has the same variance in every dimension, and high dimensional observations are extremely helpful. \cite{chierichetti2014learning} initiated the study of this problem and obtained (in Theorem 5.2) an algorithm that with probability $1 - \Theta(1/n)$, it holds that $\Ex[|\muh_i - \mu_i|] = O \left( \sqrt{\frac{1}{\sum_{j=2}^n \frac{1}{\sigma_j^2}}} \right)$ when $d = \Omega(\log(n))$. Note how this quantity is exactly the error for estimation with known-variances, other than the removal of the term depending on $\sigma_1$. The crux of their approach leverages that with $d = \Omega(\log(n))$ dimensions, one can approximate $\sigma_i^2 + \sigma_j^2$ well for every pair of $i \ne j$. 

Interestingly, we will obtain similar guarantees while only requiring $d \ge 2$. We provide a high-level overview focusing on the most interesting case of $d=2$. Let us denote the known-variance error ignoring $\sigma_1$ as $R(\sigma) \triangleq  \sqrt{\frac{1}{\sum_{i=2}^n \frac{1}{\sigma_i^2}}} $. We note its relation to a simpler closed-form:

\begin{lemma}
    $\min_{2 \le i \le n} \frac{\sigma_i}{\sqrt{i}} \le \tilde{O}(R(\sigma))$. 
\end{lemma}

Establishing this simpler closed-form as our goal, we sketch an approach based on balance-testing that may hope to obtain error near $\frac{\sigma_i}{\sqrt{i}}$:

\begin{itemize}
    \item Consider a guess for the mean $\muh = \muh_1,\muh_2$.
    \item Filter all $X_j$ whose observation in the first dimension is farther than $\sigma_i$ from $\muh_1$.
    \item With the filtered points in the second dimension, perform balance testing around $\muh_2$. 
\end{itemize}

Informally, consider how often a sample $X_j$ with large $\sigma_j$ would ``interfere'' with a balance test at the scale of $\sigma_i$ in the 1-dimensional setting: it would land in $[\mu - \sigma_i,\mu + \sigma_i]$ with probability $\Theta(\frac{\sigma_i}{\sigma_j})$. However, in the 2-dimensional setting, this probability is much smaller given our filtering, and is accordingly $\Theta\left(\left(\frac{\sigma_i}{\sigma_j}\right)^2\right)$. This difference will be enough to obtain known-variance rates. Algorithmically, we will try all $O(n^2)$ possible filterings, each creating an instance of 1-dimensional estimation, and we will intersect all the confidence intervals yielded from each instance to obtain an estimate. 

For some intuition regarding why we obtain known-variance rates, consider the case where $i^* = \argmin_{C' \log^2(n) \le i \le n} \frac{\sigma_{i}}{\sqrt{i}}$. We claim that (after some calculation)
the conditions of \cref{claim:findit} under which we expect to find balance 
are satisfied when $\Delta \ge \frac{\log(n)\sigma_{i^*}}{\sqrt{i^*}}$. Accordingly, there exists a $C'$ such that if $i^* \ge C' \log^2(n)$ then we obtain error $\tilde{O}(\frac{\sigma_{i^*}}{\sqrt{i^*}})$ with high probability. Handling other cases where $i^* < C' \log^2(n)$ involve other considerations that ultimately yield:

\theoremestmultiboth*

\section{Estimation Error Guarantees}

In this section, we will establish our core results and techniques in three main thrusts. First, in \cref{sub:formal}, we formally introduce the concept of balance.  In \cref{subsub:alg}, we discuss our algorithm.

Second, we will show that balance is well-behaved. In \cref{subsub:unif}, we prove uniform convergence bounds that imply (i) no false balance will exist with high probability, and (ii) all good balances will exist with high probability. Further, in \cref{subsub:flexible}, we show how balance is well-behaved with flexibility towards small perturbations of the testing tuple, and obtain that if there exists a good balance then our small set of testing tuples will also find a similar balance. In total, the well-behaved nature of balance will enable us to focus on just showing the existence of a desirable balance with high probability.

Third, we accordingly focus on showing the existence of desirable balance with high probability. In \cref{sub:exist}, we show high-probability existence for balance in the Subset-of-Signals model.

Finally, in \cref{sub:combine}, we combine these thrusts to prove the estimation guarantees of \cref{theorem:est-large-n,theorem:est-small-n}. We also include a note on how to remove parameters from the algorithm, at the cost of a slower running time.

\subsection{Formalizing Balance-Finding}\label{sub:formal}

In \cref{algo:test}, we formally define testing balance. Our definition includes a parameter $C_{\dfp}$ that intends to be set such that for any desired error probability $\frac{1}{n^{\dfp}}$, a particular balance where $|\mu - \muh| > \Delta$ will pass with probability as most $\frac{1}{n^{\dfp}}$.

We now show there exists a $C_{\dfp}$ that satisfies our desired property:

\begin{lemma}\label{lemma:no-false-pos}
    For any constant $\dfp$, there exists another constant $C_{\dfp}$ such that any particular $(w,\Delta,\muh)$-balance test where $|\mu - \muh|>\Delta$ will pass with probability $\le \frac{1}{n^{\dfp}}$.
\end{lemma}
\begin{proof}
    Without loss of generality, suppose $\muh > \mu + \Delta$. We will show that it is very likely the right half will not be sufficiently larger when shifting to the left (and thus would fail on \cref{line:shift-left-fail}). Intuitively, this is because $\Ex[L_{\textrm{shift-left}}] > \Ex[R_{\textrm{shift-left}}]$. Note how the balance test will fail on \cref{line:shift-left-fail} if $R_{\textrm{shift-left}} - L_{\textrm{shift-left}} \le \sqrt{C_{\dfp} \log(n) (L_{\textrm{shift-left}} + R_{\textrm{shift-left}})}$. Let $\textrm{Rad}(T)$ denote a random variable that is the sum of $T$ Rademacher random variables. Accordingly, the probability of passing is then bounded by:
    \begin{align}
        & \Pr\left[R_{\textrm{shift-left}} - L_{\textrm{shift-left}} > \sqrt{C_{\dfp} \log(n) T_{\textrm{shift-left}}} \right] \label{step:def-prob}\\ 
        & \le \max_{T'} \Pr\left[R_{\textrm{shift-left}} - L_{\textrm{shift-left}} > \sqrt{C_{\dfp} \log(n) T'} |  T_{\textrm{shift-left}}=T' \right] \\
        & \le \max_{T'} \Pr\left[\textrm{Rad}(T') > \sqrt{C_{\dfp} \log(n) T'}\right] \\
        & \le 2 \exp\left(\frac{-2 \cdot C_{\dfp} \log(n) T'}{4T'}\right) = 2 n^{-\frac{C_{\dfp}}{2}}
    \end{align}
    Thus, we get our desired guarantee of $2 n^{-\frac{C_{\dfp}}{2}} \le n^{-\dfp}$ which can be attained by $C_{\dfp} = 2 \dfp + 2$ when $n\ge2$.
\end{proof}

We similarly aim to define conditions under which a desirable balance will fail with probability at most $\frac{1}{n^\dfn}$.

\begin{definition}\label{def:goodness}
    A tuple $(w,\Delta,\mu)$ is $C_1$-good if it satisfies $\Ex[\cnt(\mu-w,\mu+w)]\ge C_1 \log(n)$ and $\Ex[\cnt(\mu,\mu+\Delta)-\cnt(\mu + w, \mu + w + \Delta)] \ge \sqrt{C_1 \log(n) \Ex[\cnt(\mu-w,\mu+w)]}$.
\end{definition}

In terms of these conditions:
\begin{lemma}\label{lemma:no-false-neg}
    For any constants $\dfn,C_{\dfp}$, there exists a constant $C_1$ such that for any $(w,\Delta,\mu)$-balance that is $C_1$-good, it will fail with probability $\le \frac{1}{n^{\dfn}}$.
\end{lemma}
\begin{proof}
    We will separately show that each shift will fail with probability at most $\frac{1}{2 n^{\dfn}}$, and then by union bound can conclude our desired result. Without loss of generality, consider the left shift. We will consider our realization in two stages: (i) we realize $\Tls$, and then (ii) we realize $\Rls - \Lls$ by the sum of $\Tls$ variables over $\{\pm 1 \}$.

    We begin by showing that $\Tls$ will concentrate within a factor of $2$ with probability $1 - \frac{1}{4 n^{\dfn}}$, by Chernoff bound. Let us define this event as $E_{T}$:
    \begin{align}
        & \Pr[E_{T}] = 1 - \Pr[\Tls > 2 \Ex[\Tls] ] - \Pr[\Tls < \frac{1}{2} \Ex[\Tls]] \\
        & \ge  1- \Pr[|\Tls - \Ex[\Tls] | \ge \frac{1}{2} \Ex[\Tls]] \ge 2 \exp(- \Ex[\Tls] / 12) \\
        & \ge 1-  2 \exp( - C_1 / 12 \log(n)) = 2 n^{-C_1 / 12}
    \end{align}

    This is at least $1 - \frac{1}{4 n ^{\dfn}}$ when $C_1 \ge 12 \dfn + 24$ and $n \ge 2$. Now, we will show the balance fails the left shift with probability at most $\frac{1}{2 n^{\dfn}}$. Let $D$ be a random variable corresponding to the sum of $n$ i.i.d. random variables that are $+1$ if it corresponds to the sample being in the left range, $-1$ if the sample corresponds to being in the right range, and $0$ otherwise. Additionally, let $D_{\textrm{nonzero}(t)}$ denote the sum of $t$ i.i.d. random variables that are identical to the variables of $D$ conditioned on being nonzero. Note how $\Rls - \Lls = D = D_{\textrm{nonzero}}(\Tls)$. Finally, let $Z(\Tls)$ be the transformation of $D_{\textrm{nonzero}}(\Tls)$ from $\pm 1$ to $0/1$, meaning each $-1$ is mapped to $0$ and $1$ is mapped to $1$. Accordingly, $D_{\textrm{nonzero}}(\Tls) = 2 Z(\Tls) - \Tls$. Then:

    \begin{align}
        & \Pr[\Rls - \Lls \le \sqrt{C_{\dfp} \log(n) \Tls}] \\
        & = \Pr[D \le  \sqrt{C_{\dfp} \log(n) \Tls}] \\ 
        & \le \Pr[E_T] + \max_{\frac{1}{2} \Ex[\Tls] \le t \le 2 \Ex[\Tls]} \Pr[D \le \sqrt{C_{\dfp} \log(n) \Tls} | \Tls = t] \label{step:def-neg-prob}\\
        & = \Pr[E_T] + \max_{\frac{1}{2} \Ex[\Tls] \le t \le 2 \Ex[\Tls]} \Pr[D_{\textrm{nonzero}}(t) \le \sqrt{C_{\dfp} \log(n) t}] \\
        & = \Pr[E_T] + \max_{\frac{1}{2} \Ex[\Tls] \le t \le 2 \Ex[\Tls]} \Pr[ 2 Z(t) - t \le \sqrt{C_{\dfp} \log(n) t}] \\
        & = \Pr[E_T] + \max_{\frac{1}{2} \Ex[\Tls] \le t \le 2 \Ex[\Tls]} \Pr[ Z(t) \le \frac{\sqrt{C_{\dfp} \log(n) t} + t}{2}] \\
        & \le \Pr[E_T] + \max_{\frac{1}{2} \Ex[\Tls] \le t \le 2 \Ex[\Tls]} \Pr[|Z(t) - \Ex[Z(t)]| > \sqrt{C_2 C_{\dfp} \log(n) t}] \label{step:use-big-exp} \\ 
        & = \Pr[E_T] + \max_{\frac{1}{2} \Ex[\Tls] \le t \le 2 \Ex[\Tls]} \Pr[|Z(t) - \Ex[Z(t)]| > \frac{\sqrt{C_2 C_{\dfp} \log(n) t}}{\Ex[Z(t)]} \Ex[Z(t)]]  \\
        & \le \Pr[E_T] + \max_{\frac{1}{2} \Ex[\Tls] \le t \le 2 \Ex[\Tls]} \Pr[|Z(t) - \Ex[Z(t)]| > \frac{\sqrt{C_2 C_{\dfp} \log(n)}}{\sqrt{t}} \Ex[Z(t)]]  \\ 
        & \le \Pr[E_T] + \max_{\frac{1}{2} \Ex[\Tls] \le t \le 2 \Ex[\Tls]} 2 \exp\left( -\Ex[Z(t)] \left( \frac{\sqrt{C_2 C_{\dfp} \log(n)}}{\sqrt{t}} \right)^2 / 3 \right)\\
        & \le \Pr[E_T] + \max_{\frac{1}{2} \Ex[\Tls] \le t \le 2 \Ex[\Tls]} 2 \exp\left(  -\frac{t}{2} \cdot \left( \frac{\sqrt{C_2 C_{\dfp} \log(n)}}{\sqrt{t}} \right)^2 / 3 \right)\\
        & \le \Pr[E_T] + \max_{\frac{1}{2} \Ex[\Tls] \le t \le 2 \Ex[\Tls]} 2\exp\left(  -C_2 C_{\dfp} \log(n) / 6 \right) \\
        & = \Pr[E_T] + 2 n^{-\frac{C_2 C_{\dfp}}{6}} \\
        & \le \frac{1}{4 n ^{\dfn}}+ 2 n^{-\frac{C_2 C_{\dfp}}{6}}
    \end{align}

    This quantity is bounded by $\frac{1}{2 n^{\dfn}}$ if $C_2 = \frac{6 \cdot (\dfn + 3)}{C_{\dfp}}$ and $n \ge 2$. \cref{step:use-big-exp} holds by noting how $\Ex[Z(t)] = \frac{\Ex[D_{\textrm{nonzero}}(t)] + t}{2}$ and how:
    
    \begin{align}
        & \Ex[D_{\textrm{nonzero}}(t)] \\
        & \ge \frac{1}{2} \Ex[D] \\
        & \ge \frac{1}{2} \sqrt{C_1 \log(n) \Ex[\cnt(\mu-w,\mu+w)]}  \\
        & \ge (1 + 2 \sqrt{C_2}) \sqrt{C_{\dfp} \log(n) \cdot 4 \cdot \Ex[\cnt(\mu-w,\mu+w)]}  \label{step:sub-c1}\\
        & \ge (1 + 2 \sqrt{C_2}) \sqrt{C_{\dfp} \log(n) \cdot 2 \cdot \Ex[\Tls]} \\
        & \ge (1 + 2 \sqrt{C_2}) \sqrt{C_{\dfp} \log(n) \cdot t}
    \end{align}
    \cref{step:sub-c1} holds for sufficiently large $C_1$ where $C_1 \ge 16 (1 + 2\sqrt{C_2})^2 C_{\dfp}$. Thus, we fail the left shift with probability at most $\frac{1}{2 n^{\dfn}}$, and by union bound with the right shift we fail the balance test with probability at most $\frac{1}{n^{\dfn}}$ as desired.
\end{proof}

Accordingly, we have formalized our balance test, and provided lemmas that configure parameters to have desired false positives and false negatives under particular conditions.

\subsubsection{Balance-Finding Algorithm}\label{subsub:alg}

Our algorithm will involve testing many collections of $(w,\Delta,\muh)$ balances. For a given $w$ and $\Delta$, we design an sweep-line algorithm that tests all values of $\muh$, and returns the ranges of $\muh$ for which the balance test passes. 

Without loss of generality, consider determining the values of $\muh$ for which the left shift passes the $(w,\Delta,\muh)$ balance test. We will compute this with a sweep-line approach, and start with $\muh = -\infty$. At this point, every sample $X_i$ is contributing to none of $\Tls,\Lls,\Rls$. When $\muh = X_i - w - \Delta$, it starts contributing to $\Tls$ and $\Lls$, then at $\muh = X_i - \Delta$, it swaps its contribution to $\Tls$ and $\Rls$, and finally at $\muh = X_i - \Delta + w$ it no longer contributes to anything. Accordingly, as we sweep from left to right, there are $O(1)$ events to process for each $X_i$. Between events, the values of $\Tls,\Lls,\Rls$ remain constant and thus whether the balance test passes also remains the same. This gives us $O(n)$ events to process for evaluating whether both shifts pass as we sweep, and accordingly an $O(n \log(n))$ time algorithm that returns at most $O(n)$ intervals corresponding to the values of $\muh$ for which the $(w,\Delta,\muh)$ balance test passes. This corresponds to \cref{line:sweep} of \cref{algo:estimate}.

Moreover, if there are no false positives, we use the existence of a $(w,\Delta,\muh)$ balance to yield confidence that $\mu \in [\muh - \Delta, \muh + \Delta]$. Our final algorithmic aspect will be to select the values of $w$ and $\Delta$ for which to use our sweep-line testing. Later, we will observe that it is sufficient to try values of $w$ and $\Delta$ that are either $\infty$ or approximately some $\sigma_{2^i}$. While we are not told the value of $\sigma_{2^i}$, we can show that with high probability it is within a polynomial factor of the gap between the closest $2^i$ samples, and thus we can approximate this by one of $O(\log(n))$ powers of $2$ near the gap between the closest $2^i$ samples. This corresponds to \cref{line:params} of \cref{algo:estimate}.

We outline this entire approach in \cref{algo:estimate}. The algorithm runs in $\tilde{O}(n)$ time because we enumerate over $\tilde{O}(1)$ configurations of parameters and test each configuration with our sweep-line algorithm in $\tilde{O}(n)$ time.  We additionally note how $\operatorname{Generate-Tests}$ does provide a correct approximation of each $\sigma_{2^i}$ with high probability:

\begin{lemma} \label{lemma:pows-ok}
    For any constant $\dpa$, there exists a constant $C_{\dpa}$ such that $\spa$ contains a value $x \in [\sigma_{2^i},2 \sigma_{2^i}]$ for every $i$ with probability at least $1 - \frac{1}{n^{\dpa}}$.
\end{lemma}
\begin{proof}
    Note how this claim holds if every $r_{2^i}$ is within a factor of $n^{C_{\dpa}}$ of $\sigma_{2^i}$. 
    
    \underline{$r_k$ is not too small.} For a particular $r_k$ to be too small, it must be the case that at least one of $X_k, \dots, X_n$ is within $\frac{\sigma_k}{n^{C_{\dpa}}}$ of another sample. To bound the probability of such an event for an $X_j$, consider this event as $E_{j-\textrm{close}}$, and we will realize the other samples first and then just consider the probability that $X_j$ is realized within $\frac{\sigma_{k}}{n^{C_{\dpa}}}$ of any other sample. This attains $\sum_{k \le j \le n} \Pr[E_{j-\textrm{close}}] \le n \cdot O\left(\frac{n \cdot \frac{\sigma_k}{n^{C_{\dpa}}}}{\sigma_j}\right) \le O\left(\frac{1}{n^{C_{\dpa}-2}}\right)$. If we union bound over all values of $k$, the probability of any $r_k$ being too small is bounded by $\frac{1}{2 n^{\dpa}}$ if $C_{\dpa} > \dpa + 4$ and $n \ge 2$.

    \underline{$r_k$ is not too large. } For any $r_k$ to be too large, it must be the case that at least one of $X_i$ is farther than $\frac{1}{2} \cdot \sigma_i \cdot n^{C_{\dpa}}$ from the mean. By Chebyshev's inequality and union bound, this probability is bounded by $\frac{4}{n^{2 C_{\dpa} - 1}}$. Thus, this probability is bounded by $\frac{1}{2 n^{\dpa}}$ if $C_{\dpa} > \frac{\dpa}{2} + 2$ and $n \ge 2$.

    Accordingly, by union bound on both cases, we obtain our desired guarantee.
\end{proof}

\begin{corollary} \label{corr:pows-ok}
    For any constant $\dpa$, there exists a constant $C_{\dpa}$ such that $\spa$ contains a value $x \in [n^{c} \cdot \sigma_{2^i},2 n^{c} \cdot \sigma_{2^i}]$ for every $i$, and $-2 \le c \le 2$ with probability at least $1 - \frac{1}{n^{\dpa}}$.
\end{corollary}
\begin{proof}
    This follows immediately from invoking \cref{lemma:pows-ok}, and adding $2$ to the obtained $C_{\dpa}$.
\end{proof}

\subsection{Well-Behaved Properties of Balance}

We have now defined an algorithm that tests many balances and has proven conditions under which it has desirable false-positives and false-negatives for a particular test. In \cref{subsub:unif}, we show uniform convergence guarantees that let us bound false-positives and false-negatives for the infinite collection of balances our sweep-line tests. 
In \cref{subsub:flexible}, we show that multiplicatively perturbing $w,\Delta$ does not dramatically affect the conditions of a balance test. 

\subsubsection{Uniform convergence-like guarantees} \label{subsub:unif}

Our algorithm will test many balances, and our hope is that with high probability all tests will be correct. For our confidence intervals to be valid, all $(\cdot,\Delta',\muh)$-balance with $\mu \notin [\muh-\Delta',\muh+\Delta']$ must fail the balance test. Similarly, if we expect to see particular balances with high probability, we hope to witness such balances among the ones we test. Given that our sweep-line algorithm tests an infinite collection of tuples, and that our tuples are not chosen independently of the samples\footnote{This issue could be resolved by splitting the data in half, although this is unnecessary and uniform convergence will allow for a simpler algorithm.}, we elect to show uniform convergence guarantees. 

Let us consider how a particular sample affects a test for $(w,\Delta,\mu)$-balance. How the sample affects the evaluation of the left shift is purely a function of whether it is in $(\mu-\Delta-w,\mu-\Delta)$, $(\mu-\Delta,\mu-\Delta+w)$, or neither. Similarly, how the sample affects the evaluation of the right shift is purely a function of whether it is in $(\mu+\Delta-w,\mu+\Delta)$, $(\mu+\Delta,\mu+\Delta+w)$, or neither. In total, there are at most 9 possibilities\footnote{This is a loose upper bound, but the looseness will not affect our results.} for how a sample affects testing the balance, and the outcome for the test is purely a function of the number of samples of each possibility. Accordingly, testing a $(w,\Delta,\mu)$-balance can be viewed as applying a function $f: \mathcal{X} \rightarrow \mathcal{Y}$ to each $X_i$ where $|\mathcal{Y}|=9$, and then evaluating a function $V: (\mathcal{X},\mathcal{Y})^n \rightarrow \{0,1\}$. Similar to traditional uniform convergence guarantees, our goal is to show that the function class $\mathcal{F}$ has limited expressiveness in a way that is helpful for generalization guarantees.

For each of the $9$ elements in the range of a balance-testing function $f$, its preimage must be a contiguous range of the sorted inputs. This implies a strong condition on the possible labels of $\mathcal{Y}$:

\begin{condition}[$k$-representative labeling]
A function class $\mathcal{F}: \mathcal{X} \rightarrow \mathcal{Y}$ satisfies $k$-representative labeling if for any $X \in \mathcal{X}^n$ and $f \in \mathcal{F}$, there exists a subset $L_f \subseteq [n], |L_f|\le k$ such that any $g \in \mathcal{F}$ where $f(X_i)=g(X_i) \, \forall i \in L_f$, must satisfy $f(X_i)=g(X_i) \, \forall i \in [n]$. 
\end{condition}

In other words, this condition means that knowing the label of a particular $k$ elements in $X$ must determine the labels for all remaining points. Function classes that correspond to a collection of functions for testing balance must satisfy this property:

\begin{claim}
    Suppose $\mathcal{F}$ is a function class corresponding to the functions of a subset of balance testing tuples. Then, $\mathcal{F}$ satisfies $18$-representative labeling.
\end{claim}
\begin{proof}
    As mentioned earlier, the preimage of any element in the domain of $f \in \mathcal{F}$ must be a contiguous range of the sorted samples. For any $y \in \mathcal{Y}$ with non-empty preimage, let us add the first and last element of its preimage to our subset. This will determine the remaining labels because any unlabelled point must be between two samples with the same evaluation of $f$, and since the preimage is a contiguous range, this sample must have the same evaluation of $f$ as those two samples. In total, our label set $|L_f| \le 2 \cdot |\mathcal{Y}| \le 18$.
\end{proof}

We will now build towards a lemma that lets us conclude something informally like the following:

\begin{center}
\emph{Consider a set $\mathcal{F}_{bad}$ of functions for balance tests such that, with high probability, a particular $f \in \mathcal{F}_{bad}$ would not satisfy the balance even if $18$ samples were adversarially modified. Then, with high probability, none of the $f \in \mathcal{F}_{bad}$ will satisfy balance.}
\end{center}

As $|\mathcal{F}_{bad}|$ may be infinite, our result achieves something that an immediate union-bound does not attain. We will be able to show our balance-testing satisfies the claims of this conclusion we are building towards, as well as a similar result for a set of $\mathcal{F}_{good}$ functions corresponding to balance tests satisfying \cref{claim:findit}. First, let us state our general lemma that powers this goal:

\begin{lemma}\label{lemma:k-rep}
    Suppose we have a function class $\mathcal{F}: \mathcal{X} \rightarrow \mathcal{Y}$ satisfying $k$-representative labeling. Additionally, let $Y_{\nu}$ denote all modifications of $Y$ that change at most $\nu$ entries, and consider a function $V: \mathcal{Y}^n \rightarrow \mathcal{U}$. If $X_i$ are mutually independent, then:

    \begin{align*}
        & \Pr_X[\exists f \in \mathcal{F} s.t. V(f(X))=u] \le (n |\mathcal{Y}|)^{k} \cdot \max_{f \in \mathcal{F}} \Pr_X[V(f(X)_{k})=u]
    \end{align*}
\end{lemma}
\begin{proof}
Note that ``$V(f(X)_k)=u$'' is overloaded notation that refers to the event where there is \emph{any} modification of $k$ entries of $f(X)$ such that the value of $V(\cdot)$ after the modification is $u$. For a subset $L \subseteq[n]$, let $X_L,Y_L$ denote the list of $X$ and $Y$ with indices corresponding to the subset $L$. Similarly, let $X_{-L}$ denote the list of $X$ with indices that are not in the subset $L$. Let $f^{fit}_{X_L,Y_L}(X)$ denote an arbitrary $f$ such that $f(X_L(i)) = Y_L(i)$ for every $i \in [|L|]$ (if such an $f$ exists). Then:
    \begin{align}
        & \Pr_X[\exists f s.t. V(f(X))=u] \\
        & \le \sum_{L \in [n]^k,Y_L \in \mathcal{Y}^k} \Pr_X[\mathbbm{1}_{\exists f s.t. L_f=L, f(X_L)=Y_L} \cdot V(f(X))=u] \\
        & = \sum_{L \in [n]^k,Y_L \in \mathcal{Y}^k} \Pr_X [\mathbbm{1}_{\exists f s.t. L_f=L, f(X_L)=Y_L} \cdot V(f^{fit}_{X_L,Y_L}(X))=u)] \\ 
        & = \sum_{L \in [n]^k,Y_L \in \mathcal{Y}^k} \Pr_{X_L}[\mathbbm{1}_{\exists f s.t. L_f=L, f(X_L)=Y_L}] \cdot \Pr_{X_{-L}}[V(f^{fit}_{X_L,Y_L}(X))=u) | X_L=x_L] \\ 
        & \le \sum_{L \in [n]^k,Y_L \in \mathcal{Y}^k} \Pr_{X_L}[\mathbbm{1}_{\exists f s.t. L_f=L, f(X_L)=Y_L}] \cdot \max_{f \in \mathcal{F}} \Pr_{X_{-L}}[V(f(X))=u) | X_L=x_L] \\ 
        & \le \sum_{L \in [n]^k,Y_L \in \mathcal{Y}^k} \max_{x_L \in \mathcal{X}^k} \max_{f \in \mathcal{F}} \Pr_{X_{-L}}[V(f(X))=u) | X_L=x_L]] \\ 
        & = \sum_{L \in [n]^k,Y_L \in \mathcal{Y}^k}  \max_{f \in \mathcal{F}} \max_{x_L \in \mathcal{X}^k} \Pr_{X_{-L}}[V(f(X))=u) | X_L=x_L]] \\ 
        & \le \sum_{L \in [n]^k,Y_L \in \mathcal{Y}^k} \max_{f \in \mathcal{F}} \Pr_{X}[V(f(X)_{k})=u] \\ 
        & \le (n |\mathcal{Y}|)^{k} \cdot \max_{f \in \mathcal{F}} \Pr_{X}[V(f(X)_{k})=u] 
    \end{align}
\end{proof}

Finally, all we need to do is analyze $\max_{f \in \mathcal{F}} \Pr_{X}[V(f(X)_{18}) = u]$ for our function classes $\mathcal{F}_{bad}$ and $\mathcal{F}_{good}$. Let $V$ be the function that denotes whether balance passes or fails according to \cref{algo:test} ($V(\cdot)=1$ indicates balance passes).

\begin{lemma}\label{lemma:no-wrong}
    Let $\mathcal{F}_{bad}$ denote the set of all $(w,\Delta,\muh)$-balances where $\mu \notin [\muh - \Delta, \muh + \Delta]$. For any constant $\dnfp$, there exists a constant $C_{\dfp}$ such that \underline{all} $(w,\Delta,\muh)$-balance tests in $\mathcal{F}_{bad}$ will fail with probability at least $1- \frac{1}{n^{\dnfp}}$.
\end{lemma}

\begin{proof}
    First, we note that the probability of any particular test in $\mathcal{F}_{bad}$ passing is bounded, with 18 samples perturbed:

    \begin{claim} \label{claim:bad-small}
    There exists a $C_{\dfp}$ such that $\max_{f \in \mathcal{F}_{bad}} \Pr_X[V(f(X)_{18})=1] \le \frac{1}{9^{18} n^{\dnfp + 18}}$.
    \end{claim}
    \begin{proof}
        The proof is the same as \cref{lemma:no-false-pos}. We briefly outline the minor difference. To deal with the 18 modifications, in \cref{step:def-prob}, the probability of passing could instead be bounded by:

        \begin{align}
        & \Pr\left[R_{\textrm{shift-left}} - L_{\textrm{shift-left}} > \sqrt{C_{\dfp} \log(n) (T_{\textrm{shift-left}}-18)} - 18 \wedge T_{\textrm{shift-left}} \ge C_{\dfp} \log(n) - 18 \right]\\ 
        & \le \max_{T' \ge C_{\dfp} \log(n) - 18} \Pr\left[R_{\textrm{shift-left}} - L_{\textrm{shift-left}} > \sqrt{C_{\dfp} \log(n) (T'-18)} - 18 |  T_{\textrm{shift-left}}=T' \right] \\
        & \le \max_{T' \ge C_{\dfp} \log(n) - 18} \Pr\left[\textrm{Rad}(T') > \sqrt{C_{\dfp} \log(n) (T'-18)} - 18 \right] \\
        & = \max_{T' \ge C_{\dfp} \log(n) - 18} \Pr\left[\textrm{Rad}(T') > \frac{1}{2}\sqrt{C_{\dfp} \log(n) T'} \right]  \label{step:use-big-n}\\
        & \le 2 \exp\left(\frac{-2 \cdot C_{\dfp} \log(n) T'}{16T'}\right) = 2 n^{-\frac{C_{\dfp}}{8}}
    \end{align}
    \cref{step:use-big-n} holds for sufficiently large $n$, and thus there exists a constant $C_{\dfp}$ where the above quantity is bounded by $\frac{1}{9^{18} n^{\dnfp + 18}}$.
    \end{proof}

    Thus, by \cref{lemma:k-rep} and \cref{claim:bad-small}, the probability that any of $\mathcal{F}_{bad}$ passing is at most $\frac{1}{9^{18} n^{\dnfp + 18}} \cdot (9n)^{18} \le \frac{1}{n^{\dnfp}}$.
\end{proof}

\begin{lemma}\label{lemma:all-good-are-good}
    Let $\mathcal{F}_{good}$ denote the set of all $(w,\Delta,\muh)$-balances that are $C_1$-good. For any constants $\dnfn,C_{\dfp}$, there exists a constant $C_1$ such that \underline{all} tests in $\mathcal{F}_{good}$ pass with probability at least $1 -  \frac{1}{n^{\dnfn}}$.
\end{lemma}
\begin{proof}
    First, we note that the probability of any particular test in $\mathcal{F}_{good}$ failing is bounded, with 18 samples perturbed:

    \begin{claim} \label{claim:bad-neg}
    There exists a $C_1$ such that $\max_{f \in \mathcal{F}_{good}} \Pr_X[V(f(X)_{18})=0] \le \frac{1}{9^{18} n^{\dnfn + 18}}$.
    \end{claim}
    \begin{proof}
        The proof is the same as \cref{lemma:no-false-neg}. The only change is to modify \cref{step:def-neg-prob} in the same manner as was done in \cref{claim:bad-small}.
    \end{proof}

    Thus, by \cref{lemma:k-rep} and \cref{claim:bad-neg}, the probability that any of $\mathcal{F}_{good}$ failing is at most $\frac{1}{9^{18} n^{\dnfn + 18}} \cdot (9n)^{18} \le \frac{1}{n^{\dnfn}}$.
    
\end{proof}

\subsubsection{Flexibility of Balance Parameters}\label{subsub:flexible}

Recall how we determine parameters $w$ and $\Delta$ in $\operatorname{Generate-Tests}$ (\cref{line:params} of \cref{algo:estimate}) by including powers of $2$ such that one is near a desired $w$ and $\Delta$ with high probability by \cref{lemma:pows-ok}. We now show that such an approximation does not worsen the conditions of a good balance as defined in \cref{lemma:no-false-neg} beyond a factor of $2$:

\begin{lemma}\label{lemma:perturb-params}
    Consider a tuple $(w,\Delta,\mu)$ such that is $C_1$-good. Then, for any $\Delta' \in [\Delta,2 \Delta]$ and $w' \in [w, 2w]$, it must be $C_1/2$-good.
\end{lemma}
\begin{proof}
    For the first condition:
    \begin{align}
        & \Ex[\cnt(\mu - w', \mu + w')] \\
        & \ge \Ex[\cnt(\mu - w, \mu + w)] \\
        & \ge C_1 \log(n) \ge \frac{C_1}{2} \log(n)
    \end{align}

    For the second condition:
    \begin{align}
        & \Ex[\cnt(\mu,\mu + \Delta') - \cnt(\mu + w', \mu + w' + \Delta')] \\
        & = (\Ex[\cnt(\mu,\mu + \Delta) - \cnt(\mu + w', \mu + w' + \Delta)]) + (\Ex[\cnt(\mu + \Delta,\mu + \Delta') - \cnt(\mu + w' + \Delta, \mu + w' + \Delta')]) \\
        & \ge (\Ex[\cnt(\mu,\mu + \Delta) - \cnt(\mu + w', \mu + w' + \Delta)]) \\
        & \ge (\Ex[\cnt(\mu,\mu + \Delta) - \cnt(\mu + w, \mu + w + \Delta)]) \\
        & \ge \sqrt{C_1\log(n) \Ex[\cnt(\mu - w, \mu + w)]} \\
        & \ge \sqrt{\frac{C_1}{2} \log(n) \Ex[\cnt(\mu - w', \mu + w')]}
    \end{align}
\end{proof}

\subsection{Existence of Balance}\label{sub:exist}

With the previous results, we may now focus on just showing the existence of a desirable balance with high probability. We now show the desired balances in the (unsimplified) Subset-of-Signals model. 

\emph{\underline{Larger $m$.}} For large $m$, the proof will be near-identical to the proof of \cref{lemma:toy-large-n} for the simplified model:

\begin{lemma}\label{lemma:real-large-n}
    When $m \in [n^{1/4},n]$, then for any constant $C_1$ there exists a $(w,\Delta,\mu)$-balance with $\Delta=\tilde{O}\left( \frac{n}{m^4} \right)^{1/6}$ that is $C_1$-good.
\end{lemma}
\begin{proof}
    We will consider evaluating two types of balance, and conclude that at least one of these balances must be $C_1$-good with the desired $\Delta$. 
    
    $(1,\Delta,\mu)$-balance is $C_1$-good if $\Omega(1) \cdot (m \cdot \Delta)^2 \ge C_1 \log(n) \cdot \Ex[\cnt(\mu-1,\mu+1)]$. Otherwise, it must hold that $\Delta \le O(1) \cdot \sqrt{\frac{C_1 \log(n) \Ex[\cnt(\mu - 1, \mu + 1)]}{m^2}}$.

    Intuitively, if this is an undesirable bound on $\Delta$, then $\Ex[\cnt(\mu - 1, \mu + 1)]$ must be large, meaning many of the $n-m$ samples must be realized in $[-1,+1]$, and this will force the median to perform well. In other words, either we are able to find balance from our $m$ ``good'' samples, or our remaining $n-m$ ``bad'' samples must not actually be too bad. For our other type of balance, we will notice how $(\infty,\Delta,\mu)$-balance behaves similarly to classical high-probability guarantees for the median. This balance will be $C_1$-good if $\Ex[\cnt(\mu-\Delta,\mu+\Delta)] \ge O(1) \cdot \sqrt{C_1 \log(n) n}$.

    Combining both restrictions, if neither balance is $C_1$-good then $\Delta \le O(1) \cdot \sqrt{\frac{C_1 \log(n) \Ex[\cnt(\mu - 1, \mu + 1)]}{m^2}} \le O(1) \cdot \sqrt{\frac{C_1 \log(n) \Ex[\cnt(\mu - \Delta, \mu + \Delta)]}{\Delta m^2}} \le O(1) \cdot \sqrt{\frac{C_1 \log(n) \sqrt{C_1 \log(n) n}}{\Delta m^2}}$. This implies $\Delta \le O(1) \cdot (C_1^{3/2} \log^{3/2}(n))^{1/3} \cdot (\frac{n}{m^4})^{1/6} = O(\sqrt{\log(n)} \cdot (\frac{n}{m^4})^{1/6}) = \tilde{O}((\frac{n}{m^4})^{1/6}))$. 
    
\end{proof}

\emph{\underline{Smaller $m$.}} For small $m$, the proof will have more nuance to deal with complications from samples that can disrupt balance at the scale of the $m$ good $\sigma_i$, but the samples do not force the median to perform well:

 \begin{lemma}\label{lemma:real-small-n}
    For any constant $C_1$ there exists a constant $C'$ such that when $m \in [C' \log(n),n^{1/4}]$, there exists a $(w,\Delta,\mu)$-balance with $\Delta=\tilde{O}\left( \frac{n}{m^4} \right)^{1/2}$ that is $C_1$-good.
\end{lemma}
\begin{proof}
    Instead of considering a pair of balances, we will now consider a richer collection of balances. Still, we will conclude that at least one of these options must exhibit goodness with the desired $\Delta$. 

    At a high-level, we will show that one of the following cases must always hold:

    \begin{enumerate}
        \item $(1,\frac{1}{2},\mu)$-balance is good. We will show this must occur if $\Omega(1)$ fraction of the samples landing in $[-1,+1]$ are from the $m$ variance-bounded points (meaning $\frac{m}{\Ex[\cnt(-1,+1)]} = \Omega(1)$). 
        \item The median-like balance of $(\infty,\Delta,\mu)$ is good. We will show this must occur if $(1,\frac{1}{2},\mu)$-balance is not good and $\Omega(1)$ fraction of the samples landing in $[-1,+1]$ are from samples with $\sigma_i \ge \Delta$.
        \item $(2^j, 2^{j-1},\mu)$-balance is good for a $1 \le 2^j \le \Delta$. We will show this must occur if none of the previous cases hold. 
    \end{enumerate}

    \textbf{Case 1. } If $(1,\frac{1}{2},\mu)$ balance is not good, then it must hold that $\Omega(1) \cdot m^2 \le C_1 \log(n) \Ex[\cnt(-1,+1)]$. Then, we know $\frac{m}{\Ex[\cnt(-1,+1)]}\le O(1) \cdot \frac{C_1 \log(n)}{m} \le O(1) \cdot \frac{C_1}{C'}$.

    \textbf{Case 2. } We will show that if $(1,\frac{1}{2},\mu)$-balance is not good, then either $(\infty,\Delta,\mu)$ is good or it must hold that the fraction of expected samples landing in $[-1,+1]$ from $\sigma_i \ge \Delta$ is bounded by a constant of our choice. Recall how if $(1,\frac{1}{2},\mu)$-balance is not good then $\Omega(1) \cdot m^2 \le C_1 \log(n) \Ex[\cnt(-1,+1)]$, and thus $\Ex[\cnt(-1,+1)] \ge \Omega(1) \cdot \frac{m^2}{C_1 \log(n)}$. Let $\rho_{\ge \Delta}(l,r)$ denote the number of samples in $[l,r]$ from samples with $\sigma_i \ge \Delta$. Accordingly, if for some constant $C_{\Delta}$ it holds that $\frac{\Ex[\rho_{\ge \Delta}(-1,+1)]}{\Ex[\cnt(-1,+1)]} \ge C_{\Delta}$, then it must hold that $\Ex[\rho_{\ge \Delta}(-\Delta,+\Delta)] \ge \Omega(1) \cdot \Delta \cdot \Ex[\rho_{\ge \Delta}(-1,+1)] \ge \Omega(1) \cdot C_{\Delta} \cdot \Delta \cdot \Ex[\cnt(-1,+1)] \ge \Omega(1) \cdot C_{\Delta} \cdot \Delta \cdot \frac{m^2}{C_1 \log(n)}$. Recall how $(\infty,\Delta,\mu)$-balance is $C_1$-good if $\Ex[\rho_{\ge \Delta}(0,\Delta)] \ge \sqrt{C_1 \log(n) n}$, which is implied if $\Omega(1) \cdot C_{\Delta} \cdot \Delta \cdot \frac{m^2}{C_1 \log(n)} \ge \sqrt{C_1 \log(n) n}$ and thus by $\Delta \ge O(1) \cdot \frac{C_1^{1.5} \log^{1.5}(n) \sqrt{n}}{C_{\Delta} m^2} \ge \frac{1}{C_{\Delta}} \tilde{O}((\frac{n}{m^4})^{1/2})$. So, either we find our desired balance or we have shown our desire that $\frac{\Ex[\rho_{\ge \Delta}(-1,+1)]}{\Ex[\cnt(-1,+1)]} \le C_{\Delta}$. Arbitrarily, we set $C_{\Delta}=\frac{1}{2}$ to conclude $\frac{\Ex[\rho_{\ge \Delta}(-1,+1)]}{\Ex[\cnt(-1,+1)]} \le \frac{1}{2}$

    \textbf{Case 3. } Finally, we will show that if neither of the above cases held, then there must be a $(2^j, 2^{j-1},\mu)$-balance for a $1 \le 2^j \le \Delta$. To accomplish this, we will be reasoning about the density of samples that have roughly similar variances. Let us define $\rho_i(a,b)$ as the number of samples in $(a,b)$ from samples with $\sigma_j \in (2^{i-1},2^i]$. As a special case, let $\rho_0(a,b)$ denote the number of samples in $(a,b)$ from samples with $\sigma_j \le 1$. We will look at the value that corresponds to having the highest density in $[-1,+1]$: $i^* \triangleq \argmax_i \Ex[\rho_i(-1,+1)]$. We will define some quantities about the density with regards to the density at $i^*$: $b(j) \triangleq \frac{\Ex[\rho_j(-1,+1)]}{\Ex[\rho_{i^*}(-1,+1)]}$ and $B_{\Delta}(j) \triangleq \sum_{k>j}^{\lceil \log(\Delta) \rceil} b(j)$. Then for any $j$:

    \begin{align}
        & \Ex[\cnt(-1,+1)] \\
        & = \sum_{k=0}^{\infty} \Ex[\rho_k(-1,+1)] \\
        & \le 2 \cdot \sum_{k=0}^{\lceil \log(\Delta)\rceil} \Ex[\rho_k(-1,+1)] \label{step:use-case-2} \\
        & = 2 \cdot \left(\sum_{k=1}^j \Ex[\rho_k(-1,+1)] + \sum_{j+1}^{\lceil \log(\Delta) \rceil} \Ex[\rho_k(-1,+1)]\right) \\
        & = 2 \cdot \left(\sum_{k=1}^j \Ex[\rho_k(-1,+1)] + B_{\Delta}(j) \Ex[\rho_{i^*}(-1,+1)] \right)\\
        & \le 2 \cdot \left( j \Ex[\rho_{i^*}(-1,+1)] + B_{\Delta}(j) \Ex[\rho_{i^*}(-1,+1)] \right)
    \end{align}

    \cref{step:use-case-2} follows from the final statement within Case 2 that $\frac{\Ex[\rho_{\ge \Delta}(-1,+1)]}{\Ex[\cnt(-1,+1)]} \le \frac{1}{2}$. This implies:

    \begin{align}
        & \frac{\Ex[\rho_{i^*}(-1,+1)]}{\Ex[\cnt(-1,+1)]} \\
        & \ge \frac{1}{2 (j + B_{\Delta}(j))}
    \end{align}

    Then, we can now relate the density at $j$ compared to the total:

    \begin{align}
        & \frac{\Ex[\rho_j(-1,+1)]}{\Ex[\cnt(-1,+1)]} \\
        &= \frac{\Ex[\rho_j(-1,+1)]}{\Ex[\rho_{i^*}(-1,+1)]} \cdot \frac{\Ex[\rho_{i^*}(-1,+1)]}{\Ex[\cnt(-1,+1)]} \\
        & \ge \frac{b(j)}{2(j + B_{\Delta}(j))}
    \end{align}

    Our hope will be to show now that either we can find $(2^j,2^{j-1},\mu)$-balance, or $b(j)$ must be small in such a way that $B(O(1))$ is small and then we must have been able to find some $(O(1),O(1),\mu)$-balance as was done in Case 1. 

    More precisely, we can find $(2^j,2^{j-1},\mu)$-balance if $\Ex[\rho_j(-2^j,2^j)]^2 \ge O(1) \cdot C_1 \log(n) \Ex[\cnt(-2^j,2^j)]$. Moreover, we will observe this by:

    \begin{align}
        & \Ex[\rho_j(-2^j,+2^j)]^2 \\
        & \ge (\Omega(1) \cdot \Ex[\rho_j(-1,+1)] \cdot 2^j) \cdot (\frac{\Ex[\rho_j(-2^j,+2^j)]}{\Ex[\cnt(-2^j,+2^j)]} \cdot \Ex[\cnt(-2^j,+2^j)]) \\
        & \ge (\Omega(1) \cdot b(j) \cdot 2^j \cdot \Ex[\rho_{i^*}(-1,+1)]) \cdot (\frac{b(j)}{2(j+B_{\Delta}(j))} \cdot \Ex[\cnt(-2^j,+2^j)]) \\
        & \ge (\Omega(1) \cdot b(j) \cdot 2^j \cdot m) \cdot (\frac{b(j)}{2(j+B_{\Delta}(j))} \cdot \Ex[\cnt(-2^j,+2^j)]) \\
        & \ge (\Omega(1) \cdot \frac{b(j)^2 2^j}{2(j+B_{\Delta}(j))}) \cdot (C' \log(n)  \cdot \Ex[\cnt(-2^j,+2^j)]) \\
        & = (\Omega(1) \cdot \frac{C'}{C_1} \cdot \frac{b(j)^2 2^{j-1}}{j+B_{\Delta}(j)}) \cdot (C_1 \log(n)  \cdot \Ex[\cnt(-2^j,+2^j)]) \label{step:cprime}
    \end{align}

    Meaning, this balance is $C_1$-good if $(\Omega(1) \cdot \frac{C'}{C_1} \cdot \frac{b(j)^2 2^{j-1}}{j+B_{\Delta}(j)}) \ge 1$. For a $C'$ chosen to be sufficiently large, this means the balance is $C_1$-good if, say, $\frac{b(j)^2 2^{j-1}}{j+B_{\Delta}(j)} \ge \frac{1}{100} \iff b(j) \ge \sqrt{\frac{j + B_{\Delta}(j)}{2^{j-1} \cdot 10}}$.

    Our remaining plan is as follows: (i) we notice that $B_{\Delta}(\lceil \log(\Delta) \rceil)=0$ by definition, (ii) we notice that $B_{\Delta}(-1) = \sum_{j=0}^{\lceil \log(\Delta) \rceil} \frac{\Ex[\rho_{j}(-1,+1)]}{\Ex[\rho_{i^*}(-1,+1)]} \ge \sum_{j=0}^{\lceil \log(\Delta) \rceil} \frac{\Ex[\rho_{j}(-1,+1)]}{\Ex[\cnt(-1,+1)]} \ge \frac{1}{2}$, (iii) if none of the desired balances are $C_1$-good for $j \ge 40$ then $B_{\Delta}(39)$ is small, and (iv) if $B_{\Delta}(-1) \gg B_{\Delta}(39)$ then one of $b(i)$ for $i < 40$ is large and this suffices for $(2^i,2^{i-1},\mu)$-balance to be good.

    Steps (i) and (ii) follow immediately from their statement. As mentioned above, if none of these balances are $C_1$-good then we can conclude that $b(j) \le \sqrt{\frac{j + B_{\Delta}(j)}{2^{j-1}}}$. For step (iii):

    \begin{claim}
        $B_{\Delta}(39) \le \frac{1}{4}$ if none of the balances are $C_1$-good.
    \end{claim}
    \begin{proof}
        Let us consider an inductive proof that is decreasing in $j$ until $j=39$. Suppose it holds that $B_{\Delta}(j) \le \frac{1}{4}$. Then, as none of the balances are $C_1$-good, we know $b(j) \le \sqrt{\frac{j + B_{\Delta}(j)}{2^{j-1}}} \le \sqrt{\frac{j + \frac{1}{4}}{2^{j-1}}} \le \sqrt{\frac{j}{2^{j-2}}} \le \sqrt{\frac{1}{2^{j/2-2}}} = 2 \cdot 2^{-j/4}$. If the sum of each upper bound on $\sum_{k=j}^{\lceil \log(\Delta) \rceil} b(j)$ is bounded by $\frac{1}{4}$, then our induction would hold. As expected, we observe this as $\sum_{k = j}^{\lceil \log(\Delta) \rceil} b(j) \le \sum_{k = j}^{\infty} 2 \cdot 2^{-j/4} \le \sum_{k = 40}^{\infty} 2 \cdot 2^{-j/4} < \frac{1}{4}$.
    \end{proof}

    Finally, for step (iv), note how $B_{\Delta}(-1) \ge \frac{1}{2}$ and $B_{\Delta}(39) \le \frac{1}{4}$. By pigeonhole principle, we know $\max_{0 \le j \le 39} b(j) \ge \frac{B_{\Delta}(39) - B_{\Delta}(-1)}{40} \ge \frac{1}{160}$. Let $j^*$ denote the corresponding $\argmax_{0 \le j \le 39} b(j)$ where $b(j^*) \ge \frac{1}{160}$. Then, we know $\frac{\Ex[\rho_{j^*}(-1,+1)]}{\Ex[\cnt(-1,+1)]} \ge \frac{b(j^*)}{2(j + B_{\Delta}(j^*))} \ge \frac{1/160}{2(40 + 40 + 1/4)} \ge  \Omega(1)$. We observe that $(2^{j^*},2^{j^*-1},\mu)$-balance must be $C_1$-good if 
    \begin{align}
        & \Omega(1)\cdot\Ex[\rho_{j^*}(-2^{j-1},+2^{j-1})] \ge \sqrt{C_1 \log(n) \Ex[\cnt(\mu-2^j,\mu+2^j)]} \\
        & \impliedby \Omega(1)\cdot \frac{\Ex[\rho_{j^*}(-1,+1)]}{\Ex[\cnt(-1,+1)]} \cdot \Ex[\rho_{j^*}(-1,+1)] \ge C_1 \log(n) \\
        & \impliedby \Omega(1) \cdot \frac{C' \log(n)}{160} \ge C_1 \log(n) \\
        & \impliedby C' \ge O(1) \cdot C_1
    \end{align}

    Thus, for large enough $C'$, we have shown there must exist one of the desired $C_1$-good balances.

\end{proof}

\subsection{Combining Ingredients: Obtaining an Estimation Guarantee}\label{sub:combine}

These components will be sufficient to almost immediately attain our desired estimation guarantees:

\begin{restatable}[Subset-of-Signals: Large $m$]{theorem}{theoremsublarge}\label{theorem:est-large-n}
     For any constant $\delta$, when $m \in [n^{1/4},n]$, \cref{algo:estimate} attains error $\tilde{O}\left( \frac{n}{m^4} \right)^{1/6}$ with probability at least $1 - \frac{1}{n^{\delta}}$.
\end{restatable}

\begin{proof}
    The proof follows exactly the same structure as \cref{theorem:est-small-n}. We will set the error parameters $\dnfn,\dnfp,\dpa$ to constants such that $\frac{1}{n^{\dnfn}} + \frac{1}{n^{\dnfp}} + \frac{1}{n^{\dpa}} \le \frac{1}{n^{\delta}}$. Additionally, suppose that $m$ is a power of $2$ (this can be obtained by considering the largest power of $2$ smaller than $m$). 

    By \cref{lemma:all-good-are-good}, there exists a $C_1$ such that all $C_1$-good balances pass the balance test with probability at least $1 - \frac{1}{n^{\dnfn}}$. By \cref{lemma:perturb-params}, any $2 C_1$-good balance will still be a $C_1$-good balance when its parameters are increased by at most a factor of $2$. By \cref{lemma:real-large-n}, there must exist a $(w,\Delta,\mu)$-balance with $\Delta=\tilde{O}\left( \frac{n}{m^4} \right)^{1/6}$ that is $2C_1$-good. Moreover, the $w$ and $\Delta$ for this $2C_1$-good balance are either $\infty$ or within a factor of $n$ of $\sigma_m = \sigma_{2^{\log(m)}}$, so by \cref{corr:pows-ok}, with probability at least $1 - \frac{1}{n^{\dpa}}$ we will test a $w',\Delta'$ that are within a factor of $2$ of the desired values. Thus, we will test a $C_1$-good balance with $\Delta = \tilde{O}\left( \frac{n}{m^4} \right)^{1/6}$ and it will pass with probability at least $1 - \frac{1}{n^{\dnfn}} - \frac{1}{n^{\dpa}}$. Moreover, no incorrect balance will pass with probability at least $\frac{1}{n^{\dnfp}}$, so all obtained confidence intervals are valid, and their intersection is of size $\tilde{O}\left( \frac{n}{m^4} \right)^{1/6}$ with probability $1 - \frac{1}{n^{\dnfn}} + \frac{1}{n^{\dnfp}} + \frac{1}{n^{\dpa}} \ge 1- \frac{1}{n^{\delta}}$.
\end{proof}

\begin{restatable}[Subset-of-Signals: Small $m$]{theorem}{theoremsubsmall}\label{theorem:est-small-n}
    For any constant $\delta$, there exists a constant $C'$ such that when $m \in [C' \log(n),n^{1/4}]$, \cref{algo:estimate} attains error $\tilde{O}\left( \frac{n}{m^4} \right)^{1/2}$ with probability at least $1 - \frac{1}{n^{\delta}}$.
\end{restatable}
\begin{proof}
    The proof is exactly the same as \cref{theorem:est-large-n}, with all occurrences of $\Delta$ changed to $\tilde{O}\left( \frac{n}{m^4} \right)^{1/2}$, and by invoking \cref{lemma:real-small-n} instead of \cref{lemma:real-large-n}.
\end{proof}

Together, \cref{theorem:est-large-n,theorem:est-small-n} imply \cref{theorem:sub-opt}.

\textbf{Removing parameters. } As an aside, we informally note that it is possible to get guarantees of this form, even without any parameters ($C_{\dfp}, C_{\dpa}$) appearing in the algorithm itself. First, we observe that $C_{\dpa}$ is purely used in \textsc{Generate-Tests} whose purpose is to intelligently select a set of $\tilde{O}(1)$ balances to test. Alternatively, we could more slowly test \emph{all} balances. For the same reasoning as demonstrated in \cref{subsub:unif}, there is a polynomially-bounded number of balances that perform differently on the realized samples. Consider how a right shift performs for a particular balance: a prefix of samples will be to the left of $\muh + \Delta - w$, then a contiguous range will be within $[\muh+\Delta-w,\muh+\Delta]$, then a contiguous range will be within $[\muh + \Delta, \muh + \Delta + w]$, and finally a suffix will be to the right of $\muh + \Delta + w$. Meaning, there are $O(n^3)$ possibilities for how the samples are treated with respect to the right shift. Similarly, there are $O(n^3)$ possibilities for how the samples are treated with respect to the left shift, and thus $O(n^6)$ possibilities for how the samples affect a balance test. For each possible balance test that is consistent with a particular one of these $O(n^6)$ possibilities, they will all either pass or fail for a given $C_{\dfp}$, and the confidence intervals yielded are only affected by the minimal possible consistent $\muh + \Delta$, and the maximal consistent $\muh - \Delta$ (these can be computed in $O(1)$ time). Accordingly, one algorithm that removes dependence on $C_{\dpa}$ is to try all $O(n^6)$ important balance tests, and each test is attained in a way that already knows the relevant values to be processed in $O(1)$ time, thus running in $O(n^6)$ total time. 

Further, we can also remove $C_{\dfp}$. For each possible balance test $\tau$, let $C_\tau$ be the largest value of $C_{\dfp}$ for which this balance test would pass. Note how our current algorithm performs equivalently to intersecting all confidence intervals from tests whose $C_\tau \ge C_{\dfp}$. We could instead modify our algorithm to process all balance tests in non-increasing order of $C_{\tau}$, intersecting each yielded confidence interval in this order, until we consider a confidence interval that does not intersect our current confidence interval, after which we return an arbitrary estimate in our current confidence interval. For any possible value of $C_{\dfp}$, this modified algorithm would behave consistently with the fixed-parameter version until considering $\tau$ where $C_\tau < C_{\dfp}$, after which our algorithm only would return something within the confidence interval of the fixed-parameter version, thus attaining the same guarantee. 

In summary, this modification would contain no parameters and simultaneously attain our guarantees in terms of all $\delta$, at the cost of a slower running time of $\tilde{O}(n^6)$ time.

\section{Estimation in Multiple Dimensions}

In this section, we focus on estimation with $d$-dimensional observations. Recall how each $X_i \sim N(\mu,\Sigma_i)$, where $\mu$ is a $d$-dimensional vector and $\Sigma_i$ is a $d$-dimensional covariance matrix. In particular, we focus on the setting with each $\Sigma_i = \sigma_i^2 I$. As the previously discussed result of \cite{chierichetti2014learning} (in their Theorem 5.2) attains almost known-variance rates when $d = \Omega(\log(n))$, we focus on showing that it is possible to get such rates even when $d = 2$, yielding improved rates when $d = o(\log(n))$. Formally, we define our benchmark of almost known-variance rates as:

\begin{definition}
    $R(\sigma) \triangleq  \sqrt{\frac{1}{\sum_{j=2}^n \frac{1}{\sigma_j^2}}}$
\end{definition}

Moreover, we show a relationship with this closed form that is easier to work with:

\begin{lemma}
    $\min_{2 \le 2^i \le n} \frac{\sigma_{2^i}}{\sqrt{2^i}} = \tilde{O}(R(\sigma))$
\end{lemma}
\begin{proof}
    \begin{align}
        & R(\sigma) \\\\
        & \triangleq \sqrt{\frac{1}{\sum_{i=2}^n \frac{1}{\sigma_i^2}}} \\
        & \ge \sqrt{\frac{1}{\sum_{i=1}^{\lfloor \log(n) \rfloor}\frac{2^i}{\sigma_{2^i}^2}}} \\ 
        & \ge \sqrt{\frac{1}{\log(n) \max_{2 \le 2^i \le n} \frac{2^i}{ \sigma_{2^i}^2}}} \\
        & = \frac{1}{\sqrt{\log(n)}} \cdot \min_{2 \le 2^i \le n} \frac{\sigma_{2^i}}{\sqrt{2^i}}
    \end{align}
\end{proof}

Establishing this simpler closed-form as our goal, we will design algorithms using balance-testing to obtain desired rates. In \cref{sec:multi-alg}, we introduce the modified algorithm. In \cref{sec:multi-est}, we show it attains the desired estimation guarantee in terms of $\sigma_i$ for $i \ge C' \log(n)$ with a sufficiently large $C'$. In \cref{sec:multi-small}, we adjust our algorithm to obtain desired guarantees in terms of smaller $i$ as well.

\subsection{Algorithmic Approach}\label{sec:multi-alg}

We begin by recalling the sketch of an algorithmic idea outlined in \cref{sec:over-high}:
\begin{itemize}
    \item Consider a guess for the mean $\muh = \muh_1,\muh_2$.
    \item Filter all $X_j$ whose observation in the first dimension is farther than $\sigma_i$ from $\muh_1$.
    \item With the filtered points in the second dimension, perform balance testing around $\muh_2$. 
\end{itemize}

The positive intuition about this sketch is that filtering helps make testing for balance easier. If we calculated how often a sample $X_j$ with large $\sigma_j$ would ``interfere'' with a balance test at the scale of $\sigma_i$ in the 1-dimensional setting: it would land in $[\mu - \sigma_i,\mu + \sigma_i]$ with probability $\Theta(\frac{\sigma_i}{\sigma_j})$. However, in the 2-dimensional setting, this probability is much smaller given our filtering, and is accordingly $\Theta\left(\left(\frac{\sigma_i}{\sigma_j}\right)^2\right)$. This difference is crucially what enables us to obtain known-variance rates. However, this sketch is quite vague. For example, we do not actually want to enumerate over all values of $\muh_1,\muh_2$, and we also are not told the value of $\sigma_i$ to use. 

Our concrete algorithmic idea is to consider all $O(n^2)$ possible filters. Let $Y_1^d \le Y_2^d \le \dots \le Y_n^d$ denote the sorted list of the realized samples in the $d$-th dimension. Additionally, let $\other(Y_i^d)$ denote the corresponding sample in the other dimension. For each pair of $i,j$, we will consider just the samples within $[Y_i^1,Y_j^1]$ and use our 1-dimensional algorithm among $\{ \other(Y_k^1) | Y_k^1 \in [Y_i^1, Y_j^1]\}$ to test for balance in the second dimension. If we intersect all the confidence intervals we obtained in the $O(n^2)$ instances of 1-dimensional mean estimation, we will show that we find a desirable estimate of $\mu_2$. To obtain an estimate of $\mu_1$, by symmetry we can filter based on the second dimension and estimate $\mu_1$. We formally define the procedure in \cref{algo:multi}.\footnote{An important technical detail is that inside the 1-dimensional instances, all occurrences of $n$ remain the number of samples in the original 2-dimensional instance, \emph{$n$ is not the number of samples that were passed to the 1-dimensional instance}. Additionally (this part is only for simplicity of analysis), use all samples, not just the filtered samples, for \textsc{Generate-Tests} within each subroutine.} This approach naively runs in $\tilde{O}(n^3)$ time.

\begin{algorithm}[t]
    \caption{2-Dimensional Estimation-Algorithm} \label{algo:multi}
    \hspace*{\algorithmicindent} 
    \begin{flushleft}
      {\bf Input:} $Y_1^1 \le \dots \le Y_n^1$ and $Y_1^2 \le Y_2^2 \le \dots \le Y_n^2$ \\
      {\bf Output:} Ranges  $\conf^1, \conf^2$ (can choose arbitrary estimates $\muh_1 \in \conf^1$, $\muh_2 \in \conf^2$) \\
    \end{flushleft}
    \begin{algorithmic}[1]
    
    \Procedure{Multi-Estimation-Algorithm}{$X_1 \le \dots \le X_n$}:
    \State $\conf^1 \gets [-\infty, \infty], \conf^2 \gets [-\infty,\infty]$
    \Comment{Intervals we are confident $\mu_1,\muh_2$ is within.}
    \For{$\dest \in \{1, 2\}$} \Comment{Estimating $\mu_{\dest}$.}
        \State $\dfil \gets \{1, 2 \} \setminus \dest$ \Comment{Using the other dimension $\dfil$ to filter.}
        \For{$i<j \in [n]$} \Comment{Enumerate over possible filters.}
            \State $S_{\textrm{filtered}} \gets \{ \other(Y_k^{\dfil}) | i \le k \le j\}$
            \State $\conf^{\dest} \gets \conf^{\dest} \cap \operatorname{Estimation-Algorithm}(S_{\textrm{filtered}})$ \Comment{Intersect 1-dimensional confidence.}
        \EndFor
    \EndFor
    \Return $\conf^1, \conf^2$ \Comment{Can estimate $\muh_1,\muh_2$ as any arbitrary value in $\conf^1,\conf^2$.}
    \EndProcedure
    \end{algorithmic}
\end{algorithm}

\subsection{Estimation Error}\label{sec:multi-est}

We now walk through a similar process to how we bounded estimation error in 1-dimension. Our main ingredients will be (i) showing there are still no incorrect confidence intervals, (ii) showing that one of the filters has a $C_1$-good balance with a desirable $\Delta$ for any constant $C_1$, and (iii) concluding that the algorithm will find this balance. Steps (i) and (iii) will be rather immediate, while (ii) will be the main thrust.

\underline{No incorrect balance intervals.} We begin by concluding step (i) that there are no incorrect confidence intervals:

\begin{corollary}\label{corr:no-wrong}
    For a fixed dimension, let $\mathcal{F}_{bad}$ denote the set of all $(w,\Delta,\muh)$-balances where $\mu \notin [\muh - \Delta, \muh + \Delta]$. For any constant $\dnfp$, there exists a constant $C_{\dfp}$ such that \underline{all} $(w,\Delta,\muh)$-balance tests in $\mathcal{F}_{bad}$ will fail with probability at least $1- \frac{1}{n^{\dnfp}}$.
\end{corollary}
\begin{proof}
    Note how every 1-dimensional instance is a collection of samples who are selected independent of their realization in that dimension. Thus, this claim follows immediately from \cref{lemma:no-wrong}, and adjusting $C_{\dfp}$ to provide an extra factor of $n^2$ to union bound over all instances.
\end{proof}

\underline{Existence of good balance after filtering.} For (ii), we aim to show there is a filter such that its corresponding 1-dimensional instance has a $C_1$-good balance with the desired $\Delta$:

\begin{lemma} \label{lemma:log1-bal}
    For any constant $C_1$ there exists a constant $C'$ such that there exists a filter and $(w,\Delta,\mu)$-balance with $\Delta= \tilde{O}\left(\min_{C' \log(n) \le 2^i \le n} \frac{\sigma_{2^i}}{\sqrt{2^i}}\right)$ that is $C_1$-good. 
\end{lemma}
\begin{proof}

    Let us consider trying to show the existence of a good balance with $\Delta = \tilde{O}(\frac{\sigma_{2^i}}{\sqrt{2^i}})$ for each $i$. This would imply our desired goal. For a particular $i$, we will consider the filter that only keeps samples within $[\mu_{\dfil} - \sigma_{2^i}, \mu_{\dfil} + \sigma_{2^i}]$. Now, we observe the conditions of goodness for a $(\sigma_{2^i},\Delta,\mu_{\dest})$-balance.

    For the first condition: 
    \begin{align}
        & \Ex[\cnt(\mu - \sigma_{2^i}, \mu + \sigma_{2^i})] \ge C_1 \log(n) \\
        & \impliedby \Omega(1) \cdot 2^i \ge C_1 \log(n) \\
        & \impliedby \Omega(1) \cdot C' \log(n) \ge C_1 \log(n) \\
        & \impliedby C' \ge O(1) \cdot C_1
    \end{align}

     This also holds for sufficiently large $C'$. The second condition will require slightly more machinery. We will first separately consider showing our guarantee among $2^i \ge C' \log^2(n)$, then $2^i \ge C' \log(n) \cdot \log(\log(n))$, and then finally $2^i \ge C' \log(n)$. First, for $2^i \ge C' \log^2(n)$, let $i^* = \argmin_{2^i \ge C' \log^2(n)}\frac{\sigma_{2^i}}{\sqrt{2^i}}$. Then, as long as $\Delta \le \sigma_{2^{i^*}}$ we satisfy the second condition by:

    \begin{align}
        & \Ex[\cnt(\mu,\mu + \Delta) - \cnt(\mu + \sigma_{2^{i^*}} + \Delta)] \ge \sqrt{C_1 \log(n) \Ex[\cnt(\mu - \sigma_{2^{i^*}}, \mu + \sigma_{2^{i^*}})]} \\
        & \impliedby \Omega(1) \cdot \frac{\Delta}{\sigma_{2^{i^*}}} \cdot 2^{i^*} \ge O(1) \cdot \sqrt{C_1 \log(n) \left( 2^{i^*} + \sum_{j=2^{i^*}}^{n } \left(\frac{\sigma_{2^{i^*}}}{\sigma_j} \right)^2 \right) }\\
        & \impliedby \Omega(1) \cdot \frac{\Delta}{\sigma_{2^{i^*}}} \cdot  2^{i^*} \ge O(1) \cdot \sqrt{C_1 \log(n) \left( 2^{i^*} + \sum_{j=i^*+1}^{\lfloor \log(n) \rfloor} 2^j \cdot \left(\frac{\sigma_{2^{i^*}}}{\sigma_{2^j}} \right)^2 \right) } \\
        & \impliedby \Omega(1) \cdot \frac{\Delta}{\sigma_{2^{i^*}}} \cdot  2^{i^*} \ge O(1) \cdot \sqrt{C_1 \log(n) \left( 2^{i^*} + \sigma_{2^{i^*}}^2 \log(n) \cdot \max_{i^* < j \le \lfloor \log(n) \rfloor} \frac{2^j}{\sigma_{2^j}^2} \right) }\\
        & \impliedby \Omega(1) \cdot \frac{\Delta}{\sigma_{2^{i^*}}} \cdot  2^{i^*} \ge O(1) \cdot \sqrt{C_1 \log(n) \left( 2^{i^*} + \sigma_{2^{i^*}}^2 \log(n)  \cdot \frac{2^{i^*}}{\sigma_{2^{i^*}}^2} \right) } \label{step:use-inv-fail}\\
        & \impliedby \Omega(1) \cdot \frac{\Delta}{\sigma_{2^{i^*}}} \cdot  2^{i^*} \ge O(1) \cdot \sqrt{C_1 \log^2(n) 2^{i^*}   }\\
        & \impliedby \frac{\Delta}{\sigma_{2^{i^*}}} \cdot \sqrt{2^{i^*}} \ge O(1) \cdot \sqrt{C_1 \log^2(n)}\\
        & \impliedby \frac{\frac{\sigma_{2^{i^*}}}{\sqrt{2^{i^*}}} \cdot \sqrt{C'\log^2(n)}}{\sigma_{2^{i^*}}} \cdot \sqrt{2^{i^*}} \ge O(1) \cdot \sqrt{C_1 \log^2(n)} \label{step:almost-delta2} \\
        & \impliedby \sqrt{C'} \ge O(1) \cdot \sqrt{C_1}
    \end{align}

    This holds for sufficiently large $C'$. \cref{step:almost-delta2} follows from setting $\Delta = \frac{\sigma_{2^{i^*}}}{\sqrt{2^{i^*}}} \cdot \sqrt{C' \log^2(n)}$, which is valid because $\Delta \le \sigma_{2^{i^*}}$. Moreover, this $\Delta$ satisfies our desire that $\Delta = \tilde{O}(\frac{\sigma_{2^{i^*}}}{\sqrt{2^{i^*}}})$. Thus, there is a filter with a desirable good balance with respect to $i^* \ge C' \log^2(n)$. Next, we show our guarantee for $i \ge C' \log(n) \log(\log(n))$, by considering $i^* = \argmin_{2^i \ge C' \log(n) \log(\log(n))}\frac{\sigma_{2^i}}{\sqrt{2^i}}$. Additionally, let $i'$ be the smallest integer such that $2^{i'} \ge C' \log^2(n)$. If $\frac{\sigma_{2^{i^*}}}{\sqrt{2^{i^*}}} \ge \frac{1}{\log(n)} \cdot \min_{i \ge i'}\frac{\sigma_{2^i}}{\sqrt{2^i}}$, then our guarantee already holds. Otherwise, as long as $\Delta \le \sigma_{2^{i^*}}$ we satisfy the second condition by:

    \begin{align}
        & \Ex[\cnt(\mu,\mu + \Delta) - \cnt(\mu + \sigma_{2^{i^*}} + \Delta)] \ge \sqrt{C_1 \log(n) \Ex[\cnt(\mu - \sigma_{2^{i^*}}, \mu + \sigma_{2^{i^*}})]} \\
        & \impliedby \Omega(1) \cdot \frac{\Delta}{\sigma_{2^{i^*}}} \cdot 2^{i^*} \ge O(1) \cdot \sqrt{C_1 \log(n) \left( 2^{i^*} + \sum_{j=2^{i^*}}^{n } \left(\frac{\sigma_{2^{i^*}}}{\sigma_j} \right)^2 \right) }\\
        & \impliedby \Omega(1) \cdot \frac{\Delta}{\sigma_{2^{i^*}}} \cdot  2^{i^*} \ge O(1) \cdot \sqrt{C_1 \log(n) \left( 2^{i^*} + \sum_{j=i^*+1}^{\lfloor \log(n) \rfloor} 2^j \cdot \left(\frac{\sigma_{2^{i^*}}}{\sigma_{2^j}} \right)^2 \right) } \\
        & \impliedby \Omega(1) \cdot \frac{\Delta}{\sigma_{2^{i^*}}} \cdot  2^{i^*} \ge O(1) \cdot \sqrt{C_1 \log(n) \left( 2^{i^*} + \sum_{j=i'}^{\lfloor \log(n) \rfloor} 2^j \cdot \left(\frac{\sigma_{2^{i^*}}}{\sigma_{2^j}} \right)^2     + \sum_{j=i^*+1}^{i'-1} 2^j \cdot \left(\frac{\sigma_{2^{i^*}}}{\sigma_{2^j}} \right)^2 \right) } \\
        & \impliedby \Omega(1) \cdot \frac{\Delta}{\sigma_{2^{i^*}}} \cdot  2^{i^*} \ge O(1) \cdot \nonumber \\
        & \sqrt{C_1 \log(n) \left( 2^{i^*} + \sigma_{2^{i^*}}^2 \log(n) \cdot \max_{i' \le j \le \lfloor \log(n) \rfloor} \frac{2^j}{\sigma_{2^j}^2} + \sigma_{2^{i^*}}^2 \log(\log(n)) \cdot \max_{i^{*} \le j < i'} \frac{2^j}{\sigma_{2^j}^2} \right) } \\
        & \impliedby \Omega(1) \cdot \frac{\Delta}{\sigma_{2^{i^*}}} \cdot  2^{i^*} \ge O(1) \cdot \nonumber \\
        & \sqrt{C_1 \log(n) \left( 2^{i^*} + \sigma_{2^{i^*}}^2 \log(n)  \cdot \frac{1}{\log(n)} \cdot \frac{2^{i^*}}{\sigma_{2^{i^*}}^2} +  \sigma_{2^{i^*}}^2 \log(\log(n))  \cdot \frac{2^{i^*}}{\sigma_{2^{i^*}}^2} \right) } \label{step:use-inv-fail5}\\
        & \impliedby \Omega(1) \cdot \frac{\Delta}{\sigma_{2^{i^*}}} \cdot  2^{i^*} \ge O(1) \cdot \sqrt{C_1 \log(n) \log(\log(n)) 2^{i^*}   }\\
        & \impliedby \frac{\Delta}{\sigma_{2^{i^*}}} \cdot \sqrt{2^{i^*}} \ge O(1) \cdot \sqrt{C_1 \log(n) \log(\log(n))}\\
        & \impliedby \frac{\frac{\sigma_{2^{i^*}}}{\sqrt{2^{i^*}}} \cdot \sqrt{C'\log(n) \log(\log(n))}}{\sigma_{2^{i^*}}} \cdot \sqrt{2^{i^*}} \ge O(1) \cdot \sqrt{C_1 \log(n) \log(\log(n))} \label{step:almost-delta5} \\
        & \impliedby \sqrt{C'} \ge O(1) \cdot \sqrt{C_1}
    \end{align}

    This holds for sufficiently large $C'$. \cref{step:use-inv-fail5} holds because $\frac{\sigma_{2^{i^*}}}{\sqrt{2^{i^*}}} \le \frac{1}{\log(n)} \min_{2^i \ge 2^{i'}} \frac{\sigma_{2^{i}}}{\sqrt{2^{i}}}$. \cref{step:almost-delta5} follows from setting $\Delta = \frac{\sigma_{2^{i^*}}}{\sqrt{2^{i^*}}} \cdot \sqrt{C' \log(n) \log(\log(n))}$, which is valid because $\Delta \le \sigma_{2^{i^*}}$. We still desire to show our condition with respect to $i \ge C' \log(n)$. We will consider $2^i$ in decreasing powers of $2$ from $C' \log(n) \log(\log(n))$ to $C' \log(n)$. Let us change $i'$ to be the smallest integer such that $2^{i'} \ge C' \log(n) \log(\log(n))$. 
    Then, our goal is to show that after having considered $k$ powers of $2^i$ in decreasing order, we have found a balance with $\Delta =  \tilde{O}(\min_{2^i \ge 2^{i' - k}} \frac{\sigma_{2^i}}{\sqrt{2^i}}) \cdot (\log(\log(\log(n))))^k$. Recall how we have previously shown that we find a balance with $\Delta = \tilde{O}(\min_{i \ge i'} \frac{\sigma_{2^i}}{\sqrt{2^i}})$. Thus, if $\frac{\sigma_{2^{i'-k}}}{\sqrt{2^{i'-k}}} \ge \frac{1}{\log(n)} \cdot \min_{i \ge i'} \frac{\sigma_{2^i}}{\sqrt{2^i}}$, then our guarantee already holds. 
    Additionally, if $\frac{\sigma_{2^{i'-k}}}{\sqrt{2^{i'-k}}} \ge \frac{1}{\log(\log(\log(n)))} \cdot \min_{2^i > 2^{i'-k}} \frac{\sigma_{2^i}}{\sqrt{2^i}}$, our guarantee also already holds. Otherwise, we can show the second condition holds desirably (in a matter similar to above):

    \begin{align}
        & \Ex[\cnt(\mu,\mu + \Delta) - \cnt(\mu + \sigma_{2^{i'-k}} + \Delta)] \ge \sqrt{C_1 \log(n) \Ex[\cnt(\mu - \sigma_{2^{i'-k}}, \mu + \sigma_{2^{i'-k}})]} \\
        & \impliedby \Omega(1) \cdot \frac{\Delta}{\sigma_{2^{i'-k}}} \cdot 2^{i'-k} \ge O(1) \cdot \sqrt{C_1 \log(n) \left( 2^{i'-k} + \sum_{j=2^{i'-k}}^{n } \left(\frac{\sigma_{2^{i'-k}}}{\sigma_j} \right)^2 \right) }\\
        & \impliedby \Omega(1) \cdot \frac{\Delta}{\sigma_{2^{i'-k}}} \cdot  2^{i'-k} \ge O(1) \cdot \sqrt{C_1 \log(n) \left( 2^{i'-k} + \sum_{j=i'-k}^{\lfloor \log(n) \rfloor} 2^j \cdot \left(\frac{\sigma_{2^{i'-k}}}{\sigma_{2^j}} \right)^2 \right) } \\
        & \impliedby \Omega(1) \cdot \frac{\Delta}{\sigma_{2^{i'-k}}} \cdot  2^{i'-k} \ge O(1) \cdot \nonumber \\
        & \sqrt{C_1 \log(n) \cdot \left( 2^{i'-k} + \sum_{j=i'}^{\lfloor \log(n) \rfloor} 2^j \cdot \left(\frac{\sigma_{2^{i'-k}}}{\sigma_{2^j}} \right)^2  +  \sum_{j=i'-k}^{ i'-1} 2^j \cdot \left(\frac{\sigma_{2^{i'-k}}}{\sigma_{2^j}} \right)^2  \right)} \\
        & \impliedby \Omega(1) \cdot \frac{\Delta}{\sigma_{2^{i'-k}}} \cdot  2^{i'-k} \ge O(1) \cdot \sqrt{C_1 \log(n) } \cdot \nonumber \\
        & \sqrt{ 2^{i'-k} + \sigma_{2^{i'-k}}^2 \cdot \left(\log(n) \cdot \max_{i' \le j \le \lfloor \log(n) \rfloor} \frac{2^j}{\sigma_{2^j}^2} 
 +\log(\log(\log(n))) \cdot \max_{i'-k < j < i'} \frac{2^j}{\sigma_{2^j}^2} \right)} \\
        & \impliedby \Omega(1) \cdot \frac{\Delta}{\sigma_{2^{i'-k}}} \cdot  2^{i'-k} \ge O(1) \cdot \nonumber\\
        & \sqrt{C_1 \log(n) \left( 2^{i'-k} + \sigma_{2^{i'-k}}^2 \cdot \left(\log(n) \cdot \frac{1}{\log(n)} + \log(\log(\log(n))) \cdot \frac{1}{\log(\log(\log(n)))}\right)  \cdot \frac{2^{i'-k}}{\sigma_{2^{i'-k}}^2} \right) } \label{step:use-inv-fail3}\\
        & \impliedby \Omega(1) \cdot \frac{\Delta}{\sigma_{2^{i'-k}}} \cdot  2^{i'-k} \ge O(1) \cdot \sqrt{C_1 \log(n) 2^{i'-k}   }\\
        & \impliedby \frac{\Delta}{\sigma_{2^{i'-k}}} \cdot \sqrt{2^{i'-k}} \ge O(1) \cdot \sqrt{C_1 \log(n)}\\
        & \impliedby \frac{\frac{\sigma_{2^{i'-k}}}{\sqrt{2^{i'-k}}} \cdot \sqrt{C'\log(n)}}{\sigma_{2^{i'-k}}} \cdot \sqrt{2^{i'-k}} \ge O(1) \cdot \sqrt{C_1 \log(n)} \label{step:almost-delta3} \\
        & \impliedby \sqrt{C'} \ge O(1) \cdot \sqrt{C_1}
    \end{align}

    This holds for sufficiently large $C'$. \cref{step:almost-delta3} follows from setting $\Delta = \frac{\sigma_{2^{i'-k}}}{\sqrt{2^{i'-k}}} \cdot \sqrt{C' \log(n)}$, which is valid because $\Delta \le \sigma_{2^{i'-k}}$. Moreover, this $\Delta$ satisfies our desire that $\Delta = \tilde{O}(\frac{\sigma_{2^{i'-k}}}{\sqrt{2^{i'-k}}})$. Meaning, regardless of the conditions, after processing $2^{i'-k}$ we have found a balance with $\Delta =  \tilde{O}(\min_{2^i \ge 2^{i' - k}} \frac{\sigma_{2^i}}{\sqrt{2^i}}) \cdot (\log(\log(\log(n))))^k$. After we have processed all powers of $2$ until $C'\log(n)$, we will have $k = O(\log(\log(\log(n))))$ and thus have found a balance with $\Delta = \tilde{O}(\min_{2^i \ge 2^{i' - k}} \frac{\sigma_{2^i}}{\sqrt{2^i}}) \cdot (\log(\log(\log(n))))^k = \tilde{O}(\min_{2^i \ge 2^{i' - k}} \frac{\sigma_{2^i}}{\sqrt{2^i}}) \cdot (\log(\log(\log(n))))^{O(\log(\log(\log(n))))} = \tilde{O}(\min_{2^i \ge 2^{i' - k}} \frac{\sigma_{2^i}}{\sqrt{2^i}}) \cdot \tilde{O}(1) = \tilde{O}(\min_{2^i \ge C' \log(n)} \frac{\sigma_{2^i}}{\sqrt{2^i}})$.
\end{proof}

\underline{Concluding a good estimate.} These components are enough to conclude our desirable estimation guarantee:

\begin{restatable}{theorem}{theoremestmultilarge}\label{theorem:est-multi-large}
     For $d=2$ and any constant $\delta$, there exists a $C'$ such that  with probability at least $1 - \frac{1}{n^{\delta}}$, \cref{algo:multi} obtains error $\tilde{O}\left(\min\limits_{C' \log(n) \le 2^i \le n} \frac{\sigma_{2^i}}{\sqrt{i}}\right)$.
\end{restatable}
\begin{proof}
    By \cref{corr:no-wrong}, we know there are no incorrect confidence intervals if we set $\dnfp$ accordingly relative to $\delta$. Moreover, by \cref{lemma:log1-bal}, we know that for any $C_1$ there is a $C'$ such that there is a $C_1$-good balance with a $\Delta$ having the desired guarantee with respect to $C'$. Accordingly, if we set $\dpa$ accordingly, by \cref{corr:pows-ok} we test a $C_1/2$-good balance with the desired $\Delta$, if we try the correct filter. As we test all filters, we test such a balance. Moreoever, if we set $\dnfn$ accordingly, all $C_1/2$-good balances that we test will pass. Thus, we obtain no incorrect confidence intervals, and we obtain at least one correct confidence interval with the desired width to conclude our algorithm obtains the desired error.
\end{proof}
\subsection{Handling Dependence on the Smallest Variances}\label{sec:multi-small}

While \cref{theorem:est-multi-large} only has dependence on $\sigma_i$ for $i \ge C' \log(n)$, it is still possible to obtain guarantees with dependence on the smaller terms. Note how for the desired guarantees to not automatically hold from \cref{theorem:est-multi-large}, it must be the case that $\sigma_2 \ll \sigma_{C' \log(n)}$. Accordingly, we will find it sufficient to estimate based on the closest pair of observations that are within or near $\conf^1,\conf^2$ from \cref{algo:multi}. We formalize this in \cref{algo:adjust-multi}.

\begin{algorithm}[t]
    \caption{Adjusted 2-Dimensional Estimation-Algorithm} \label{algo:adjust-multi}
    \hspace*{\algorithmicindent} 
    \begin{flushleft}
      {\bf Input:} $Y_1^1 \le \dots \le Y_n^1$ and $Y_1^2 \le Y_2^2 \le \dots \le Y_n^2$ \\
      {\bf Output:} Estimates $\muh_1,\muh_2$ \\
    \end{flushleft}
    \begin{algorithmic}[1]
    
    \Procedure{Adjusted-Multi-Estimation-Algorithm}{$X_1 \le \dots \le X_n$}:
        \State $\conf^1,\conf^2 \gets \operatorname{Multi-Estimation-Algorithm}(X)$
        \State $L \gets \max \left(\send(\conf^1)-\ssta(\conf^1),\send(\conf^2)-\ssta(\conf^2)\right)$ \Comment{Longer confidence length.}
        \State $W_1 \gets [\ssta(\conf^1) - L, \send(\conf^1) + L]$ \Comment{Ranges of confidence that contain $[\mu-L,\mu+L]$.}
        \State $W_2 \gets [\ssta(\conf^2) - L, \send(\conf^2) + L]$
        \State $D_{\textrm{best}} \gets \infty$ \Comment{Closest pair distance.}
        \State $\muh_1 \gets \ssta(W_1)$ \Comment{Arbitrary point in confidence range.}
        \State $\muh_2 \gets \ssta(W_2)$
        \For{$i<j \in [n]$}
            \If{$Y_i^1,Y_j^1 \in W_1$ \textbf{and} $\other(Y_i^1),\other(Y_j^1) \in W_2$} \Comment{Both within confidence.}
                \State $D_{i,j} \gets \max\left(|Y_i^1 - Y_j^1|,|\other(Y_i^2)-\other(Y_j^2)|\right)$
                \If{$D_{i,j} < D_{\textrm{best}}$} \Comment{This is the closest pair so far.}
                    \State $D_{\textrm{best}} \gets D_{i,j}$ \Comment{Updating closest length.}
                    \State $\muh_1 \gets Y_i^1$ \Comment{Arbitrary point of the pair.}
                    \State $\muh_2 \gets \other(Y_i^1)$
                \EndIf
            \EndIf
        \EndFor
        
    \Return $\muh_1,\muh_2$ 
    \EndProcedure
    \end{algorithmic}
\end{algorithm}

\theoremestmultiboth*
\begin{proof}
    Let us denote $\rsmall = \min_{2 \le 2^i \le C' \log(n)} \frac{\sigma_{2^i}}{\sqrt{2^i}}$, and $\rlarge = \min_{C' \log(n) \le 2^i \le n} \frac{\sigma_{2^i}}{\sqrt{2^i}}$. By definition, $R(\sigma) = \min\left(\rsmall,\rlarge\right)$.
    
    First, recall how by \cref{theorem:est-multi-large}, our ranges $\conf^1$ and $\conf^2$ both contain $\mu$ and are of length $\tilde{O}(\rlarge)$ with probability $1-o(1)$. In this event, then $W_1,W_2$ also both contain $\mu$ and are of length $\tilde{O}(\rlarge)$. As our algorithm returns an estimate within $W_1,W_2$, then our error is $\tilde{O}(\rlarge)$.

    Our main thrust is to show that the algorithm returns a quantity that is $\tilde{O}(\rsmall)$ with probability $1-o(1)$. Note that $\rsmall \ge \frac{\sigma_2}{\sqrt{C' \log(n)}}$. So, $\sigma_2 = \tilde{O}(\rsmall)$. Observe in \cref{algo:adjust-multi} that $L$ is the length of the longest interval. As we have proven that $L = \tilde{O}(\rlarge)$, then there must exist some constant $k$ such that for sufficiently large $n$ it holds $L \le \log^k(n)$. If $\rlarge \le \log^{k+4}(n) \sigma_2$ then our theorem immediately holds, meaning otherwise $\frac{\sigma_j}{\sqrt{j}} \ge \log^{k+4}(n) \sigma_2$ for every $j\ge C' \log(n)$. Let us focus on the $1-o(1)$ probability event that the subroutine of \cref{algo:multi} is correct. Note, if $L \le \log(n) \sigma_2$, our error is $\tilde{O}(\rsmall)$. Otherwise, note how $W_1$ and $W_2$ must contain the entire ranges $[\mu_1-\log(n) \sigma_2,\mu_1 + \log(n) \sigma_2]$ and $[\mu_2-\log(n) \sigma_2,\mu_2+\log(n) \sigma_2]$, respectively. With probability $1-o(1)$, the samples of $X_1$ and $X_2$ will fall within this range, and thus $D_{\textrm{best}} \le 2 \log(n) \sigma_2$. As long as we can show that all other closer pairs must be within $\tilde{O}(\sigma_2)$ of $\mu$ with probability $1-o(1)$, then our proof is complete. We handle this in four cases:

    \underline{$\min(\sigma_i,\sigma_j) \le \log^{k+5}(n) \sigma_2$.} With high probability, all such pairs will have at least one of their points within $\tilde{O}(\sigma_2)$ of $\mu$. Moreover, as $D_{\textrm{best}} \le 2 \log(n)$, both points must then be within $\tilde{O}(\sigma_2)$ of $\mu$ if this is the closest pair that determines our estimate.

    \underline{$i,j > C' \log(n)$.} Consider the event that $X_i$ lands in both $[\mu_1-2L,\mu_2+2L]$ and $[\mu_2-2L,\mu_2+2L]$, and then $X_j$ lands within $2 \log(n) \sigma_2$ of $X_i$. This must occur for the pair to be the closest valid pair. The likelihood of the first event is upper bounded by $\left(\frac{O(L)}{\sigma_i} \right)^2 \le \left( \frac{O(\sigma_i / \sqrt{i}) \cdot \log^k(n)}{\sigma_i}\right)^2 = O(\frac{\log^{2k}(n)}{i})$. Regardless of the realization of $X_i$, the likelihood that $X_j$ lands within $2\log(n) \sigma_2$ of $X_i$ is $O\left(\left(\frac{2\log(n) \sigma_2}{\sigma_j}\right)^2\right) = O\left(\left(\frac{2\log(n) \sigma_2}{\sqrt{j} \log^{k+4}(n) \sigma_2}\right)^2\right) = O\left(\frac{1}{j \log^{2k+6}(n)}
    \right)$. Thus, the probability of both occurring is $O\left(\frac{1}{i \cdot j \cdot \log^6(n)}\right)$. By union bound of all $i,j$, the probability of any such pair having these events is at most $\sum_{i,j} O\left(\frac{1}{i \cdot j \cdot \log^6(n)}\right) = o(1)$.
    
    \underline{$i < C' \log(n)$ and $\sigma_i \ge \log^{k+5}(n) \sigma_2$ and $j>C' \log(n)$.} This follows similarly to the previous case. Consider the event that $X_j$ lands in both $[\mu_1-2L,\mu_2+2L]$ and $[\mu_2-2L,\mu_2+2L]$, and then $X_i$ lands within $2 \log(n) \sigma_2$ of $X_j$. This must occur for the pair to be the closest valid pair. The likelihood of the first event is upper bounded by $\left(\frac{O(L)}{\sigma_j} \right)^2 \le \left( \frac{O(\log^{k}(n) \sigma_j / \sqrt{j}))}{\sigma_j}\right)^2 = O(\frac{\log^{2k}(n)}{j})$. Regardless of the realization of $X_j$, the likelihood that $X_i$ lands within $2\log(n) \sigma_2$ of $X_j$ is $O\left(\left(\frac{2 \log(n) \sigma_2}{\sigma_i}\right)^2\right) = O\left(\left(\frac{2 \log(n) \sigma_2}{\log^{k+5}(n) \sigma_2}\right)^2\right) = O\left(\frac{1}{ \log^{2k+8}(n)}\right)$. Thus, the probability of both occurring is $O\left(\frac{1}{j \cdot \log^8(n)}\right)$. By union bound of all $i,j$, the probability of any such pair having these events is at most \\ $\sum_{i=1}^{C' \log(n)} \sum_j O\left(\frac{1}{j \cdot \log^8(n)}\right) = o(1)$.

    \underline{$i,j < C' \log(n)$ and $\sigma_i,\sigma_j \ge \log^{k+5}(n) \sigma_2$.} Consider the event that $X_j$ lands within $2 \log(n) \sigma_2$ of $X_i$. Regardless of the realization of $X_i$, this occurs with probability at most $O\left(\left( \frac{2 \log(n) \sigma_2}{\sigma_j}\right)^2\right) = O\left(\left( \frac{2 \log(n) \sigma_2}{\log^{k+5}(n) \sigma_2}\right)^2\right) = O\left(\frac{1}{\log^{2k + 8}(n)}\right)$. By union bound of all $i,j$, the probability of any such pair having these events is at most $O(\log^2(n)) \cdot \frac{1}{\log^{2k+8}(n)} = o(1)$.

    Thus, our algorithm attains error $\tilde{O}(R(\sigma))$ with probability $1-o(1)$.
\end{proof}

\section*{Acknowledgements}

This work was supported by the National Defense Science \& Engineering Graduate (NDSEG) Fellowship Program, Tselil Schramm's NSF CAREER Grant no. 2143246, and Gregory Valiant's Simons Foundation Investigator Award.

\bibliography{ref}
\end{document}